\documentclass{article}
\usepackage[english]{babel}
\usepackage{amsfonts}
\usepackage{amsmath,amsthm,amssymb,latexsym,mathabx}
\usepackage{color}
\usepackage{authblk}
\usepackage{graphicx}
\usepackage{tikz}
\usepackage{stmaryrd}
\usepackage{mathrsfs,upgreek}
\usepackage{eucal}
\usepackage{enumerate}
\usepackage{changes}

\newtheorem{theorem}{Theorem}[section]
\newtheorem{lemma}[theorem]{Lemma}
\newtheorem{proposition}[theorem]{Proposition}
\newtheorem{corollary}[theorem]{Corollary}
\theoremstyle{definition}

\newtheorem{remark}[theorem]{Remark}

\newcommand{\op}[1]{\textrm{\upshape #1}}

\newcommand{\join}{\vee}

\newcommand{\meet}{\wedge}

\newcommand{\la}{\langle}
\newcommand{\ra}{\rangle}
\newcommand{\alg}[1]{{\textbf{\upshape #1}}}  %
\newcommand{\vv}[1]{\mathsf {#1}}

\renewcommand{\a}{\alpha}

\newcommand{\f}{\varphi}
\newcommand{\g}{\gamma}

\renewcommand{\th}{\theta}

\newcommand{\sse}{\subseteq}

\newcommand{\app}{\approx}
\newcommand{\HH}{{\mathbf H}}  
\newcommand{\SU}{{\mathbf S}} 
\newcommand{\PP}{{\mathbf P}}   
\newcommand{\VV}{{\mathbf V}}   

\newcommand{\itemb}{\item[$\bullet$]}

\newcommand{\ib}{\item[$\bullet$]}

\newcommand{\Con}[1]{\operatorname{Con}(\alg #1)}

\newcommand{\vuc}[2]{#1_1,\dots,#1_{#2}}

\newcommand{\imp}{\rightarrow}

\newcommand{\darrow}{\mathop{\downarrow}}
\makeatletter
\def\square{\RIfM@\bgroup\else$\bgroup\aftergroup$\fi
  \vcenter{\hrule\hbox{\vrule\@height.6em\kern.6em\vrule}\hrule}\egroup}
\makeatother
\newcommand{\cc}[1]{\mathsf{#1}}
\newcommand{\simA}{\mathord{\sim}}

\newcommand{\smlcirc}{\raise3pt\hbox{\textrm{\circle{3.3}}}}

\newcommand{\escl}{\text{\large$\mathop{!}$\normalsize}}

\newcommand{\myfrac}[2]{\dfrac{#1}{\lower.5ex\hbox{$#2$}}}
\mathchardef\hu="0362

\begin{document}

\title{Linear Logic via algebra I: Girard algebras}

\author{Paolo Aglian\`{o}\\
  DIISM, Universit\`a di Siena, Italy\\
agliano@live.com
\and
Mitar Nedeljkovic\\
University of Priština in Kosovska Mitrovica \\
mitar.nedeljkovic@gmail.com
}
\date{}
\maketitle

\begin{abstract} In this paper we  develop  several aspects of the algebraic theory of Girard algebras, regarded as the equivalent algebraic semantics of {\em multiplicative additive Linear Logic} $\cc{MALL}$. The main structural results concern the bottom of the subvariety lattice of the variety of Girard algebras: varieties generated by simple Girard algebras generated by the constants correspond exactly to the atoms of the lattice. Moreover the lattice itself is atomic, i.e.  every nontrivial subvariety contains an atom. These atoms can be represented uniformly by the maximal filters of the free algebra generated by the constants. Moreover every congruence of this free algebra is realized by a subvariety, and the resulting atom-support topology records the behavior of arbitrary joins: the atoms below a join of atoms are precisely the points in the closure of the corresponding subset of the spectrum.

The concrete constructions show that the bottom of the lattice is already quite rich. Besides determining the integral and degenerate boundary cases and exhibiting finite modular and nonmodular atoms and chains of covers, we constructed infinite families of atoms whose topology reflects genuinely infinitary behavior; thus finite and infinite joins behave essentially differently at the bottom of the subvariety lattice. On the logical side, these results yield $2^{\aleph_0}$ Post-complete extensions of $\cc{MALL}$.
\end{abstract}

\section{Introduction}

When one wishes to investigate  a logic  through its models, there are two main alternatives.
The first is to consider the Lindenbaum algebra of the logic and derive logical properties through algebraic
means. The other is to look at a {\em relational semantics} possibly with some duality involved.
After a promising start, whose crown achievement is probably the work of Halmos on
{\em monadic algebras} \cite{Halmos1956}, the first approach took a backseat to the second.

This was mainly due to the great success of Kripke in the sixties with his  relational semantics for
modal and intuitionistic logic that was a source of inspiration for many
researches based on his methods.  However in the seventies the algebraic approach found a new life, due mainly to the work of Blok and Pigozzi.
W. Blok, in
his Ph.D. thesis \cite{Blok1976},  conducted an
in-depth study of Lewis' modal logic S4, using {\em interior algebras} and in \cite{Blok1980} he
studied the entire {\em lattice of modal logics} by purely algebraic
means. Next Blok and Pigozzi investigated thoroughly the matter of
algebraizability of logics \cite{BlokPigozzi1989,BlokPigozzi1992}, setting the foundations
for what is now called  {\em abstract
algebraic logic}.

Their main success   was the
identification  of the ``right'' concept of {\em algebraizable logic}.
To make a long story short, a logic $\cc L$ is algebraizable if there is a class $\vv
K$
 of algebras  which is to $\cc L$ what the variety of Boolean algebras is to
classical propositional calculus.  The class $\vv K$ is called the
{\em equivalent algebraic semantics}  of $\cc L$.  This connection, dubbed the {\em Blok-Pigozzi connection}, allows one to prove many {\em bridge theorems},
showing the equivalence of logical and algebraic properties; in other words one can investigate a logic by purely algebraic means.

In the middle 90's the first author wrote a note applying this machinery to  two ``new'' logics:
{\em multiplicative additive Linear Logic} $\cc{ MALL}$ and
{\em Linear Logic} $\cc{LL}$ \cite{Girard1987}. Both turned out to be algebraizable with equivalent variety semantics: for $\cc{MALL}$ the equivalent algebraic semantics is the variety of {\em Girard algebras}
while for $\cc{LL}$ is the variety of Girales. Both are varieties of bounded $\mathsf{FL}_e$-algebras with constants and/or operations added.  The note was never published, but still attracted some attention, being an early attempt to tackle Linear Logic via algebra (another notable contribution, with a slightly different focus, is  a long draft (titled Semantical
investigations of Linear Logic) by A. Ursini, also unpublished\dots).

After about thirty years the first author, pressured by several colleagues, decided to revise the paper for publication (mostly as a historical note of the early developments of abstract algebraic logic); the paper appeared in
2025  \cite{Agliano2025}.  In 2026 the second author pointed out several errors in that paper. Although these errors do not affect its main algebraizability results, some of the arguments required a correction. While working on these corrections, the authors were led to perform a broader investigation of the algebraic models of Linear Logic. In this note we  present the resulting corrections and developments for Girard algebras, leaving the part about Girales for a subsequent paper.
Our treatment follows mostly the standard notation of universal algebra; for all unexplained concepts we refer the reader to \cite{BurrisSanka}.

\section{Preliminaries}\label{sec:preliminaries}

A {\bf commutative residuated lattice} is an algebra $\alg A=\la A,\join,\meet,\cdot,\imp, 1\ra$ where
\begin{enumerate}
\item $\la A,\join,\meet\ra$ is a lattice;
\item $\la A,\cdot,1\ra$ is a commutative monoid;
\item $\cdot$ and $\imp$ form a residuated pair w.r.t. the lattice ordering: for all $a,b,c \in A$
$$
a\cdot b \le c \qquad\text{if and only if}\qquad a \le b \imp c.
$$
\end{enumerate}
An $\mathsf{FL}_e$-algebra is an algebra $\alg A = \la A,\join,\meet,\cdot,\imp, 0,1\ra$ where
\begin{enumerate}
\item $\la A,\join,\meet,\cdot,\imp, 1\ra$ is a commutative residuated lattice;
\item $0$ is a constant operation.
\end{enumerate}
A commutative residuated lattice (or an $\mathsf{FL}_e$-algebra) is {\bf integral} if $1$ is the largest element in the ordering; an $\mathsf{FL}_e$-algebra
is {\bf zero-bounded} if $0$ is the smallest element. An $\mathsf{FL}_{ew}$-algebra is an $\mathsf{FL}_e$-algebra which is integral and zero-bounded.

On any  $\mathsf{FL}_e$-algebra one can define a unary derived operation $\simA x = x \imp 0$; an $\mathsf{FL}_e$-algebra $\alg A$  is {\bf involutive} if
$\simA\simA a =a$ for all $a \in A$.

A {\bf Girard algebra} is an algebra $\alg A =\la A,\join,\meet,\cdot,\imp, 0,1,\top,\bot\ra$ where
\begin{enumerate}
\item $\la A,\join,\meet,\cdot,\imp, 0,1\ra$ is an involutive $\mathsf{FL}_e$-algebra;
\item for all $a \in A$, $\bot \le a\le \top$.
\end{enumerate}

A {\bf girale}\footnote{note that our naming conventions are slightly different than the ones in \cite{Agliano2025} \dots} is an algebra $\alg A=\la A,\join,\meet,\cdot,\imp, 0,1,\top,\bot,\escl\ra$ where
\begin{enumerate}
\item $\la A,\join,\meet,\cdot,\imp, 0,1,\top,\bot\ra$ is a Girard algebra;
\item $\escl$ is unary and for all $a,b \in A$
\begin{align*}
&\escl 1 = 1 \tag{G1}\\
&\escl a \le a \meet 1 \tag{G2}\\
&\escl a  \escl b = \escl(a \meet b) \tag{G3}\\
&\escl\escl a = \escl a. \tag{G4}
\end{align*}
\end{enumerate}

Girard algebras and girales form varieties of algebras which we denote by $\mathsf{GA}$ and $\mathsf{Gi}$ respectively. Since Girard algebras are commutative residuated lattices augmented by constants, their structure theory is the same as that of commutative residuated lattices.
Commutative residuated lattices are congruence permutable and congruence point-regular at $1$, and also congruence distributive (since they have lattice operations). It follows that they are {\em ideal determined} \cite{GummUrsini1984} and have a good theory of ideals in the sense of \cite{AglianoUrsini1992}. The role of the ideals is played by the  filters: a {\bf filter} $F$ of $\alg A$ is a subset of $A$ that is a lattice filter containing $1$ and is closed under multiplication.
Filters  form an algebraic lattice; in particular, if $A^+ = \{a: a \ge 1\}$, then for any $\th \in \Con A$,
$$
A^+/\th = \bigcup\{a/\th: a \ge 1\}
$$
is a  filter. Moreover the two mappings
$$
\th \longmapsto A^+/\th \qquad\qquad F \longmapsto \th_F = \{(a,b): a\imp b, b\imp a \in F\}
$$
are mutually inverse order preserving maps from  $\Con A$ to the filter lattice of $\alg A$, that we shall write as $\op{Fil}(\alg A)$. Thus  $\Con A$ and $\op{Fil}(\alg A)$ are  isomorphic as lattices. Notice that $A^+/\th$ defines the same set as $\uparrow\![1]_{\theta}$, used for example in \cite{GJKO}. All these considerations apply to Girard algebras as well.

Girales are  ideal-determined as well, so their congruences are totally determined by the ideals. Now the operation
$\escl$ is a {\em compatible operation} in the sense of \cite{Agliano1996b}; this in turn implies that the ideals of a girale $\alg A$ are just the filters of its underlying Girard algebra, that are closed under $\escl$.
We will call these subsets {\em filters} as well and let the context make clear which meaning is intended. Again if $\alg A$ is a girale its congruence lattice is isomorphic with the lattice of filters.
For a girale $\alg A$ we define $\escl A= \{\escl a: a \in A\} = \{a \in A: \escl a =a\}$.

\begin{theorem} (Theorem 8.5 in \cite{Agliano2025})\footnote{The statement of Theorem 8.5 in \cite{Agliano2025} contains a typo that may obscure its intended meaning.} For any girale $\alg A$, $\escl \alg A = \la \escl A, \join, \meet_{!},\imp_{!},\bot,1\ra$ is a Heyting algebra, where for $u,v \in \escl A$
$$
u \meet_{!} v = \escl(u \meet v)\qquad  u \imp_{!}v = \escl (u \imp v).
$$
Moreover $\Con A \cong \op{Con}(\escl \alg A)$.
\end{theorem}

Two of the main results in \cite{Agliano2025} show that Girard algebras and girales are the equivalent algebraic semantics of two interesting logics. These two logics are {\em Linear Logic} $\cc{LL}$, introduced by Girard in \cite{Girard1987} and its fragment {\em multiplicative-additive Linear Logic} $\cc{MALL}$. We will not recall their definitions here; a Hilbert-style presentation of both can be found in \cite{Agliano2025}.

\begin{theorem} \cite{Agliano2025} Both $\cc{LL}$ and $\cc{MALL}$ are strongly algebraizable with equivalent variety semantics girales and Girard algebras, respectively.
\end{theorem}

The algebrization machinery in the Blok-Pigozzi style \cite{BlokPigozzi1989} allows the formulation of many {\em bridge theorems}: a logical property of a logic is reflected by an algebraic property of its equivalent algebraic semantics and vice versa. For instance if $\cc{L}$ is a strongly algebraizable logic and $\vv V$ is its equivalent variety semantics, then the lattice of subvarieties of $\vv V$ and the lattice of axiomatic extensions of $\cc{L}$ are dually isomorphic.
Similarly if $\vv V$ has {\em equationally definable principal congruences} \cite{EDPC1}, this is equivalent to saying that $\cc{L}$ has some form of {\em deduction-detachment theorem}. In \cite{Agliano2025} (Corollary 8.5) it has been shown that this happens for the variety of girales. Girard algebras do not have equationally definable principal congruences in general, but again the cases in which they do have it, can be deduced easily from the general theory of residuated lattices (see for instance \cite{Agliano1996b}). In this paper we continue the study of the algebraic properties of Girard algebras, pointing out their connections with the associated logic.

\section{Girard algebras}

\subsection{Fixing a mistake}

In \cite{Agliano2025} it was claimed that every Girard algebra is embeddable in a complete Girard algebra and this was used to show that
the variety of Girard algebras is precisely the class of subreducts (to the Girard algebra type) of girales. On the logical side this means that
$\cc{LL}$ is a conservative extension of $\cc{MALL}$.

Unfortunately the proof of the claim is incorrect: the operator that is introduced is not a closure operator at all\footnote{Indeed, $Q(\{\bot\})=\{b:\simA b\not\ge\bot\}=\varnothing$, since $\bot$ is least. Thus $Q$ is not extensive.} and the proof collapses. Luckily the claim is true nevertheless, though it requires a little more work.  Given a set $A$ and a binary relation $R$ on $A$ one can define, for $X,Y \sse A$
\begin{align*}
&X^\rhd = \{b: \text{such that}\ a\,R\, b\ \text{for all $a \in X$}\}\\
&Y^\lhd = \{a: \text{such that}\  a\,R\, b\ \text{for all $b \in Y$}\}.
\end{align*}

The proof of the following lemma is folklore.

\begin{lemma}\label{basic} Let $A$ be any set and $R$ a binary relation on $A$. Then
\begin{enumerate}
\item $\g:\mathcal P(A) \longrightarrow \mathcal P(A)$ defined by $\g(X) = X^{\rhd\lhd}$ is a closure operator;
\item $X$ is closed (i.e. $\g(X) = X$) if and only if there is a $Y \sse A$ with $X = Y^\lhd$;
\item the sets $\{b\}^\lhd$ for $b \in A$ form a base for $\g$ i.e.
$$
\g(X) = \bigcap\{\{b\}^\lhd: X \sse \{b\}^\lhd\}.
$$
\end{enumerate}
\end{lemma}

Let now $\alg A$ be an $\mathsf{FL}_{e}$-algebra. For $X,Y \sse A$ we define $XY  = \{ab: a \in X, b \in Y\}$ (and instead of $\{a\}X$ we write $aX$) and
$X \imp Y = \{a: aX \sse Y\}$. We say that a closure operator $\g$ on $\mathcal P(A)$ is a {\bf nucleus} if for all $X,Y \sse A$
\begin{equation}
\g(XY) =\g(\g(X)\g(Y)). \tag{N}
\end{equation}
\par
Then $\g$ is a {\bf downward nucleus} if every $\g$-closed set is a downset.  The following theorem is a consequence of a much more general statement about the so-called {\em power monoid construction} in \cite{GalatosJipsen2013}; we sketch a short easy proof for the reader's convenience.

\begin{theorem}\label{nuclear} Let $\alg A$ be a $\mathsf {FL}_e$-algebra and  let $\g$ be a nucleus on  $\mathcal P(A)$. Then the structure
$$
\mathcal P(\alg A)_\g=\la \mathcal P(A)_\g, \join_\g, \cap, \cdot_\g, \imp,\g(\{0\}),\g(\{1\})\ra
$$
where
$$
X \join_\g Y = \g(X \cup Y)\quad X \cdot_\g Y = \g(XY),
$$
is a complete $\mathcal{FL}_e$-algebra.

\end{theorem}
\begin{proof} (Sketch)
Since $\g$ is a closure operator $\la \mathcal P(A)_\g, \join_\g, \cap\ra$ is a complete lattice. Next using (N) one shows that
 $\cdot_\g$ is associative, while commutativity follows from commutativity in $\alg A$. That  $\g(\{1\})$ is the unit for the product is obvious.
We now show that $\mathcal P(A)_\g$ is closed under $\imp$. If $X,Y \in \mathcal P(A)_\g$ then
\begin{align*}
\g(X \imp Y)X &\sse \g(\g(X \imp Y)\g(X))\\
&= \g((X \imp Y)X)  \sse \g(Y)\\
&=Y.
\end{align*}
Hence $\g(X \imp Y) \sse X \imp Y$ and the other containment holds because $\g$ is  extensive.

Let us show that $(\cdot_\g, \imp)$ is a residuated pair. For $X,Y,Z \in \mathcal P(A)_\g$
\begin{align*}
&X \cdot_\g Y \sse Z \quad\text{if and only if}\\
&\g(XY) \sse Z \quad\text{if and only if}\\
&XY \sse Z\quad\text{if and only if}\\
&\text{for all $x \in X$, $xY \sse Z$} \quad\text{if and only if}\\
&\text{for all $x \in X$, $ x \in Y \imp Z$} \quad\text{if and only if}\\
&X \sse Y \imp Z.
\end{align*}

\end{proof}

If $\alg A$ is any $\mathsf{FL}_e$-algebra, let us define a relation $N$ on $\alg A$ by setting
$$
(a,b) \in N\quad\text{if and only if}\quad ab \le 0.
$$

\begin{lemma} The closure operator $\g$ on $\alg A$, associated to  the relation $N$, is a downward nucleus on $\mathcal P(A)$.
\end{lemma}
\begin{proof} Let $X,Y \sse A$;  we will show that
$$
\g(X)\g(Y) \sse \g(XY)
$$
since, using extensivity and monotonicity,  this yields (N). To do that we must show that if $a \in \g(X)$, $b\in \g(Y)$ and $c \in (XY)^\rhd$,
$abc \le 0$.   Now for every $x \in X$,  $y \in Y$ one has  $xyc \le 0$ so $xc \in Y^\rhd$ for all $x \in X$. As $b \in (Y^\rhd)^\lhd$  it follows that
$bxc\le 0$ for all $ x \in X$ and thus $bc \in X^\rhd$. Finally since $a \in (X^\rhd)^\lhd$ we get $abc \le 0$ as wished.

To show that $\g$ is downward we have to show that every $\g(X)$ is a downset. By Lemma  \ref{basic}
$$
\g(X) = \bigcap\{\{b\}^\lhd: X \sse \{b\}^\lhd\}.
$$
and $\{b\}^\lhd = \downarrow(\simA  b)$. Hence they are all downsets and so is $\g(X)$, being an intersection of downsets.
\end{proof}

\begin{theorem} Let $\alg A$ be a Girard algebra, $N$ the binary nuclear relation described before and $\g$
its associated nucleus. Then $\mathcal P(\alg A)_\g$ is a Girard algebra. Hence every Girard algebra is embeddable in a complete Girard algebra.
\end{theorem}
\begin{proof} By Theorem \ref{nuclear},  for the first assertion we need only show that $\g(\{0\})$ is an involutive element. In other words if we set
$\simA X = X \imp \g(\{0\})$, we must show that $\simA\simA X = X$ for every $X \in \mathcal P(A)_\g$. But since the inequality
$x\le (x \imp y) \imp y$ holds in any commutative residuated lattice we need only show that $\simA\simA X \sse X$.

For every $a\in A$, residuation and involutivity give $\{a\}^\rhd=\downarrow\!\simA a$ and $\g(\{a\})=(\downarrow\!\simA a)^\lhd=\downarrow a$. In particular $\g(\{0\}) = \darrow\{0\}$ and $\g(\{1\}) = \darrow\{1\}$. Given this it is easily seen that $u \in \simA\simA X$ if and only if
$$
\text{for all $a \in A$, $Xa \sse \darrow\{0\}$ implies $ua \le 0$}.
$$
Now let's write $X $ as $ \bigcap_{i \in I} \darrow\{\simA b_i\}$; for all $i \in I$ and $x \in X$ we have $x \in \darrow\{\simA  b_i\}$ i.e.
$x \le \simA  b_i$ which implies $xb_i \le 0$. Since $x$ was generic we have that for all $i \in I$, $Xb_i \sse \darrow \{0\}$. Since $u \in \simA\simA X$
by the above property $ub_i \le 0$ and so $u \le \simA b_i$ for all $i$. But this implies that $u \in \bigcap_{i \in I}\darrow\{\simA  b_i\} = X$, as wished.

It remains to verify the embedding. The map $a\mapsto\g(\{a\})=\downarrow a$ is injective and preserves the operations:
\begin{align*}
\downarrow a\cap\downarrow b&=\downarrow(a\meet b),\\
\g(\downarrow a\cup\downarrow b)&=\downarrow(a\join b),\\
\g((\downarrow a)(\downarrow b))&=\downarrow(ab),\\
(\downarrow a)\imp(\downarrow b)&=\downarrow(a\imp b).
\end{align*}
For joins, use $\{a,b\}^\rhd=\downarrow(\simA a\meet\simA b)$ and involutivity; the other equalities follow from the principal-closure formula and residuation. The constants $0,1$ are preserved by that formula, and $A=\downarrow\top$. Finally, $\g(\varnothing)=A^\lhd=\{\bot\}$: $x\in A^\lhd$ implies $x\top\le0$, hence $\top\le\simA x$ and $x=\bot$; the converse follows from $\bot a=\bot\le0$.
\end{proof}

\begin{corollary}\label{cor:directgirale}
Every Girard algebra admits a girale expansion on the same universe. Consequently the class of Girard-algebra reducts of girales is exactly $\mathsf{GA}$; in particular $\cc{LL}$ is a conservative extension of $\cc{MALL}$.
\end{corollary}
\begin{proof}
On a Girard algebra define
\[
\escl_H a=\begin{cases}1,&1\le a,\\ \bot,&1\nleq a.\end{cases}
\]
Then (G1) and (G2) are immediate. For (G3), both sides are $1$ when $1\le a,b$, and otherwise both are $\bot$. The values $1,\bot$ are fixed, giving (G4) (the trivial algebra causes no exception). The conservativity statement follows from the strong algebraizability theorem stated above.
\end{proof}

 \subsection{$0$-generated algebras}

 In this section we will start our investigation of  the lattice of subvarieties of Girard algebras, which we will denote by $\Lambda(\mathsf{GA})$. We will start our investigation with the atoms of the lattice.

  \begin{remark} Beside being a natural starting point, atoms in $\Lambda(\mathsf{GA})$ correspond to coatoms in the lattice of axiomatic extensions of $\cc{MALL}$ (see the end of Section \ref{sec:preliminaries}).  Those are exactly the {\em Post-complete} extensions of $\cc{MALL}$, i.e. the consistent extensions having no consistent proper extensions. Our characterization of the atoms (Theorem \ref{thm:realinterval} below) will imply that there are $2^\aleph_0$  such extensions.
  \end{remark}

 It is a well-known result by B. J\'onsson \cite{Jonsson1967} (and we often refer  to it as {\em J\'onsson Lemma}) that if $\alg A$ is an algebra in a congruence distributive variety, then all the subdirectly irreducible members of $\VV(\alg A)$ belong to
 $\HH\SU\PP_u(\alg A)$. Clearly if $\alg A$ is finite this means that they all belong to $\HH\SU(\alg A)$.  Moreover we define an algebra to be {\bf strictly simple} if it is  simple and has no proper subalgebras.

Since Girard algebras are congruence distributive the previous fact applies.
The following proposition has been proved in \cite{Galatos2005} for residuated lattices, but the same proof goes through for Girard algebras.

\begin{proposition} A variety $\vv V$ is a finitely generated atom in $\Lambda(\mathsf{GA})$ if and only if it is generated by a finite strictly simple algebra.
\end{proposition}

 Observe also that in any Girard algebra $a \imp b= \simA(a\cdot \simA b)$; so we have only to consider the lattice structure and the operations $\simA$ and $\cdot$. Now, because of involutivity and the bounds, in any Girard algebra $\simA 0=1$ (hence $\simA 1=0)$, $\simA \top =\bot$ (hence $\simA \bot =\top$) and $\bot \cdot a = \bot$ for all $a$.
Hence the multiplication table of the constants is almost determined:
\begin{center}
\begin{tabular}{c|cccc}
  $\cdot$ & $\bot$ & $0$ & $1$ & $\top$ \\
  \hline
  $\bot$ & $\bot$ & $\bot$ & $\bot$ & $\bot$  \\
  $0$ & $\bot$ & $??$ & $0$ & $??$ \\
  $1$ & $\bot$& $0$ & $1$ & $\top$  \\
  $\top$ & $\bot$ & $??$ & $\top$ & $\top$  \\
\end{tabular};
\end{center}
the only free spots are $0 \cdot 0$ and $0 \cdot \top$. Moreover it is obvious that a Girard algebra has no proper subalgebras if and only if  it is generated by the constants. Given our observation above we say that a Girard algebra is {\bf $0$-generated} if it is generated by the constants; hence the finitely generated atoms in $\Lambda(\mathsf{GA})$ ($\Lambda(\mathsf{Gi})$) are exactly those generated by finite, simple and $0$-generated Girard algebras (girales). It also follows from general results about congruence distributive varieties by B. J\'onsson that the correspondence between finite simple and 0-generated Girard algebras and finitely generated atoms in $\Lambda(\mathsf{GA})$ is a bijection up to isomorphism.

For Girard algebras, however, we can remove finiteness and keep the same outcome. Let $\alg A$ be a Girard algebra; we denote by $\alg A_0$ the $0$-generated subalgebra of $\alg A$.

\begin{lemma}\label{lem:freezero}
If $\mathbf A$ is a  Girard algebra, then
\[
\alg A_0\cong \alg F_{\VV(\alg A)}(\emptyset).
\]
\end{lemma}

\begin{proof}
Observe that there is  a canonical
homomorphism
$h:\alg F_{\mathsf{GA}}(\emptyset)\to\alg A$, whose image is
$\alg A_0$.

If $s,t$ are two terms in  $\alg F_{\mathsf{GA}}(\emptyset)$, since they contain only constants,  $h(s) = h(t)$ if and only if $s^\alg A = t^\alg A$ if and only if $\alg A \vDash s \app t$.  This means that $\op{ker}(h)$ is exactly the congruence defining $\alg F_{\VV(\alg A)}(\emptyset)$, from which the conclusion follows.
\end{proof}

\begin{corollary}\label{cor:freezero} If $\alg S$ is a $0$-generated Girard algebra then, for every $\alg B\in\VV(\alg S)$ there is a unique
homomorphism $\alg S\to\alg B$, and its image is $\alg B_0$.
\end{corollary}
\begin{proof} If $\alg S$ is $0$-generated, then $\alg S_0  =\alg S$; hence by Lemma \ref{lem:freezero}
\[
\alg S\cong \alg F_{\VV(\alg S)}(\emptyset).
\]
The rest follows
from the universal property of $\alg F_{\VV(\alg S)}(\emptyset)$.
\end{proof}

\begin{lemma}\label{lem:nontrivial} If $\alg A$ is a Girard algebra and $\th$ is a proper congruence of $\alg A$, then $(1,\bot) \notin \th$. Hence if
$\alg A$ is nontrivial, then $\bot \ne 1$ in $\alg A$.
\end{lemma}
\begin{proof} If $(1,\bot) \in \th$, then  for $a \in A$, $a= a\cdot1\mathrel{\th} a\cdot\bot=\bot$. Hence $\th$ is the total congruence.
\end{proof}

For any variety  $\vv V$ of Girard algebras we denote by $ \op{At}(\vv V)$ the set of atoms in $\Lambda(\vv V)$.

\begin{theorem}\label{thm:atomchar}
The assignment $\mathbf S\mapsto\VV(\alg S)$ is a bijection
\[
\bigl\{\text{simple $0$-generated Girard algebras}\bigr\}/{\cong}
\;\longrightarrow\;
 \op{At}(\mathsf{GA}).
\]
Every nontrivial subvariety of $\mathsf{GA}$ contains an atom, i.e. $\Lambda(\mathsf{GA})$ is atomic.
\end{theorem}

\begin{proof}
Let $\alg S$ be simple and $0$-generated, let $\vv W$ be nontrivial and let
$\vv W\subseteq\VV(\alg S)$. Choose a nontrivial
$\alg B\in\vv W$. By Corollary \ref{cor:freezero} there is a homomorphism
$g:\alg S\to\alg B$. Its kernel is $0_\alg S$ or $1_\alg S$. It cannot
be $1_\alg S$, because its image is $\alg B_0$, which contains the distinct elements $\bot^{\alg B}$ and
$1^{\alg B}$ by Lemma \ref{lem:nontrivial}. Hence $g$ is injective, so
$\alg S\in\vv W$ and $\vv W =\VV(\alg S)$. Thus
$\VV(\alg S)$ is an atom.

Conversely, let $\vv V$ be a nontrivial subvariety of $\mathsf{GA}$. Then $\alg F_\vv V(\emptyset)$ is nontrivial;  by Lemma \ref{lem:nontrivial} the
union of a chain of proper congruences is proper, so Zorn's lemma gives a
maximal proper congruence. Let $\mathbf S$ be the corresponding simple
quotient of $\alg F_\vv V(\emptyset)$. Then $\mathbf S$ is $0$-generated and belongs to $\vv V$, so,
 by the first part, $\VV(\alg S)$ is an atom contained in $\vv V$. This proves atomicity; if $\vv V$ itself is an atom, then $\VV(\alg S)=\vv V$.

Finally, if two simple $0$-generated algebras $\mathbf S,\mathbf T$ generate
the same variety, Lemma \ref{lem:freezero} identifies both with the
variety's free algebra on the empty set, so $\alg S\cong\alg T$.
\end{proof}

\subsection{Atoms and (relative) spectra}

Since Girard algebras are ideal determined, we can use filters instead of congruences (see Section \ref{sec:preliminaries}). If $\alg A$ is a Girard algebra and $F$ is a filter of $\alg A$ we will simply write $\alg A/F$ for $\alg A/\th_F$.

\begin{lemma}\label{lem:relativespectra} Let $\vv V$ be any variety of Girard algebras. Then the function  $F \longmapsto \VV(\alg F_\vv V(\emptyset)/F)$ is a bijection between the
set of maximal filters of $\alg F_\vv V(\emptyset)$ (which from now on we denote by $\op{MaxFil}(\alg F_\vv V(\emptyset))$) and $\op{At}(\vv V)$.

Consequently if $\alg A$ is a Girard algebra then the mapping $F \longmapsto \VV(\alg A_0/F)$ is a bijection between $ \op{MaxFil}(\alg A_0)$ and $\op{At}(\VV(\alg A))$.
\end{lemma}
\begin{proof}
Every maximal filter gives a simple $0$-generated quotient, hence an atom by Theorem \ref{thm:atomchar}. Conversely, the simple $0$-generated generator $\alg S$ of any atom below $\vv V$ is the image of the unique homomorphism $\alg F_\vv V(\emptyset)\to\alg S$, whose kernel is maximal. If two maximal filters give the same atom, their quotients are isomorphic by Theorem \ref{thm:atomchar}; uniqueness of the homomorphism from the free algebra then makes the two kernels, and hence the filters, equal.

The second claim follows from Lemma \ref{lem:freezero}.
\end{proof}

Thus if $\alg A_0$ is simple, $\VV(\alg A)$ has the unique atom
$\VV(\alg A_0)$; if $\alg A_0$ is finite, the number of atoms below $\VV(\alg A)$
is $|\op{MaxFil}(\alg A_0)|$.

We can improve Lemma \ref{lem:relativespectra} to a uniform characterization.
Let  $\alg F = \alg F_{\mathsf{GA}}(\emptyset)$; for a subvariety $\vv V$, let $\pi_\vv V$ be the canonical homomorphism $\alg F \to \alg F_\vv V(\emptyset)$ and let
$G_\vv V$ be the filter of $\alg F$ corresponding to the kernel of $\pi_\vv V$. Then for closed terms\footnote{ As usual, a term is {\bf closed} if it contains no variables; equivalently, its value is independent of the valuation.} $s,t$, $(s^\alg F,t^\alg F) \in \th_{G_\vv V}$ if and only if $\vv V \vDash s \app t$ and $\alg F/G_\vv V \cong \alg F_\vv V(\emptyset)$.

\begin{theorem}\label{thm:onespectrum}
For every variety $\vv V$ of Girard algebras,
\[
\op{At}( \vv V)=
\{\VV(\alg F/G):
G\in\op{MaxFil}(\alg F),\
G_\vv V\subseteq G\},
\]
and distinct $G\in\op{MaxFil}(\alg F)$ give distinct atoms.
\end{theorem}
\begin{proof} Maximal filters of
$\alg F /G_\vv V$ are exactly the
$G\in\op{MaxFil}(\alg  F)$ containing $G_{\vv V}$.
The conclusion follows by applying Lemma \ref{lem:relativespectra} .
\end{proof}

The preceding description uses only identities between closed terms. The
next theorem shows that every congruence of $\alg F$ can occur in this way;
thus there is no further restriction on these closed-term identifications.

\begin{theorem}\label{thm:realization}
For every $\th\in\Con F$, put $\vv W_\th=\VV(\alg F/\th)$. Then
\[
\th_{G_{\vv W_\th}}=\th,
\qquad F_{\vv W_\th}(\emptyset)\cong\alg F/\th.
\]
Equivalently, for every filter $D$ of $\alg F$,
$G_{\VV(\alg F/D)}=D$.
\end{theorem}
\begin{proof}
For closed terms $s,t$, the identity $s\app t$ holds in
$\VV(\alg F/\th)$ if and only if it holds in $\alg F/\th$, which is
equivalent to $(s^{\alg F},t^{\alg F})\in\th$. Every element of $F$ is
the value of a closed term, so $\th_{G_{\vv W_\th}}=\th$. The asserted
isomorphism follows from the canonical quotient map. The filter formulation
follows from the congruence--filter correspondence.
\end{proof}

\subsection{Atom supports and infinite joins}\label{sec:atomsupport}

The maximal-filter description also explains how atoms behave under joins
of varieties. From now on   $\alg F = \alg F_\mathsf{GA}(\emptyset)$ and   $X=\op{MaxFil}(\alg F)$; in view of Lemma \ref{lem:relativespectra} we will
often refer to $X$ as the {\bf atom spectrum}. For a filter $D$ of $\alg F$
\[
T(D)=\{M\in X:D\sse M\}.
\]
These sets are the closed sets of a topology $\mathcal T$ on $X$. Since $T(M)=\{M\}$ for $M\in X$, $\mathcal T$
is $T_1$ (every singleton is closed).

Now every maximal filter is meet-irreducible but
$\op{Fil}(\alg F)$ is distributive, since it is isomorphic to
$\Con F$. Hence every  maximal proper filter is meet-prime.
Consequently
\[
T\!\left(\bigvee_iD_i\right)=\bigcap_iT(D_i),
\qquad T(D\cap E)=T(D)\cup T(E).
\]
 where the join on the left-hand side is taken in the lattice of filters of $\alg F$.
Also $T(F^+)=X$ and $T(F)=\emptyset$, where $F^+=\{a\in F:1\le a\}$
is the least filter of $\alg F$.  For $Y\sse X$ its closure is
\begin{equation}\label{eq:hullclosure}
\overline Y=T\!\left(\bigcap_{M\in Y}M\right),\tag{a}
\end{equation}
with the empty intersection interpreted as $F$: a closed set $T(D)$
contains $Y$ exactly when $D\sse\bigcap_{M\in Y}M$.

\begin{theorem}\label{thm:atomsupport} Let $\alg F,X,T$ and $\mathcal T$ as above; then the following hold.
\begin{enumerate}
\item The map $\op{At}(\vv V) \longmapsto \{M \in X: \VV(\alg F/M)\in\op{At}(\vv V)\}$ is a bijection between $\{\op{At}(\vv V): \vv V \sse \mathsf{GA}\}$ and
the set of $\mathcal T$-closed subsets of $X$, so the two sets can be identified.
\item Under this identification, the map sending $\vv V \longmapsto \op{At}(\vv V)$ is a surjective and meet-complete lattice homomorphism from $\Lambda(\mathsf{GA})$ onto the lattice $\op{Cl}(X)$ of
$\mathcal T$-closed subsets of $X$.
\item For every family of atoms $(\vv A_i)_{i\in I}$,
\begin{equation}\label{eq:atomjoinclosure}
\op{At}\!\left(\bigvee_{i\in I}\vv A_i\right)
=\overline{\{\vv A_i:i\in I\}}. \tag{b}
\end{equation}
\end{enumerate}
\end{theorem}
\begin{proof}
For (1),
Theorem~\ref{thm:onespectrum} gives
$\op{At}(\vv V)=T(G_{\vv V})$. Conversely, Theorem~\ref{thm:realization}
gives $G_{\VV(\alg F/D)}=D$, so every closed set $T(D)$ occurs and the map is a bijection.

For (2) observe that
a closed identity holds in a join of varieties if and only if it holds
in every member of the family. Thus, using the congruence--filter
correspondence,
\begin{equation}\label{eq:joinfilters}
G_{\bigvee_i\vv V_i}=\bigcap_iG_{\vv V_i}.\tag{c}
\end{equation}
We have to show that
\[
\op{At}(\vv U\vee\vv V)=\op{At}(\vv U)\cup\op{At}(\vv V),
\qquad
\op{At}\!\left(\bigcap_i\vv V_i\right)=\bigcap_i\op{At}(\vv V_i).
\]
The first follows from~\eqref{eq:joinfilters} and the finite-union formula
for $T$; the second follows directly from the definition of
$\op{At}$.

Finally for (3) we observe that, if $\vv A_i=\VV(\alg F/M_i)$ is an atom, then
$G_{\vv A_i}=M_i$. Hence the left-hand side of
\eqref{eq:atomjoinclosure} is $T(\bigcap_iM_i)$, which is the required
closure by~\eqref{eq:hullclosure}.
\end{proof}

As a consequence, if $Y\sse X$ and $\alg A_M=\alg F/M$ for $M\in Y$, then, under the identification above,
\[
\op{At}\!\left(\bigvee_{M\in Y}\VV(\alg A_M)\right)=\overline{Y}.
\]
In particular, if $Y$ is closed, the atoms in this variety are exactly $\{\VV(\alg A_M):M\in Y\}$. Since $\mathcal T$ is $T_1$, every finite subset of $X$ is closed; hence finite joins introduce no new atoms. For an infinite $Y$, any additional atoms are precisely the points of $\overline{Y}\setminus Y$. The cyclic algebras in Section \ref{subsec:modular atoms} provide an explicit example (Theorem~\ref{thm:cycliclimit}).

\subsection{The shape of $\Lambda(\mathsf{GA})$}

So to describe all the atoms in $\Lambda(\mathsf{GA})$ is sufficient to describe all the maximal filters of $\alg F$.
This however is not easy in principle; $\alg F$ is huge and we do not have a sensible description so far.

Therefore to determine the shape of the bottom of the lattice $\Lambda(\mathsf{GA})$ we resort to more traditional methods. In particular we use J\'onsson's Lemma
to ``climb'' the lattice $\Lambda(\mathsf{GA})$.  A Girard algebra $\alg A$ is {\bf integral} if $1=\top$ (and thus $0=\bot$); a Girard algebra $\alg A$ is {\bf degenerate} if $\alg A\vDash 0=1$.

\begin{proposition}\label{prop:twoboundaries} Let $\alg A$ be a nontrivial Girard algebra.
\begin{enumerate}
\item If $\alg A$ is integral, then $\alg A_0$ is the two element Boolean algebra $\alg {BA}_2$; hence the variety of Boolean algebras is the only atom in $\Lambda(\mathsf{GA})$ consisting of integral Girard algebras.
\item If $\alg A$ is degenerate, then $\alg A_0$ is the three element algebra whose universe is $\{\bot,1,\top\}$ which is denoted by $\alg K_3$; hence
$\VV(\alg K_3)$ is the only atom in  $\Lambda(\mathsf{GA})$ consisting of degenerate Girard algebras.
\end{enumerate}
\end{proposition}
\begin{proof} For (1), since $1=\top$, also $0=\simA1=\simA\top=\bot$, so $\alg A_0=\{0,1\}\cong\alg{BA}_2$.

 For (2), put $e:=0=1$. Then
$\simA e=e$, $\simA\bot=\top$, $\simA\top=\bot$, and $\bot$ is absorbing.
Moreover
\[
\top=\top e\le\top^2\le\top,
\]
so $\top^2=\top$. Hence $\{\bot,e,\top\}$ is closed under the lattice
operations, multiplication and involution, and therefore also under $\imp$,
since $x\imp y=\simA(x\cdot\simA y)$. Thus
$\alg A_0\cong\alg K_3$. Moreover $\alg K_3$ is simple, since identifying
$1$ with either $\top$ or $\bot$ forces the identification of $\bot$ and
$\top$. It follows that $\VV(\alg K_3)$ is an atom and it is the unique atom
below $\VV(\alg A)$.

\end{proof}

\begin{remark} Integral Girard algebras form the equivalent algebraic semantics of the multiplicative/additive fragment of Affine Logic $\cc {LLW}$, which is $\cc {LL}$ plus the structural rule of  weakening.  This is because the algebraic translation of the weakening rule is $x \le y \imp x$ and, for Girard algebras, this is clearly equivalent to being integral. Integral Girard algebras are really equivalent to {\rm involutive $\mathsf{FL}_{ew}$ algebras}; it is well known \cite{KowalskiOno2006} that the equational theory of such algebras is decidable so the corresponding extension of $\cc{MALL}$  is a decidable logic.
\end{remark}

\begin{remark} $\alg K_3$ is the three element {\bf bounded Kalman lattice} (see \cite{AglianoMarcos2022}). For an axiomatization of bounded K-lattices see again \cite{AglianoMarcos2022}. Every bounded K-lattice can be viewed as a degenerate Girard algebra, since its dualizing element coincides with the monoid identity, which in the Girard signature is exactly the identity $0=1$. However not every degenerate Girard algebra is a bounded Kalman lattice; in the three element case it works because of cardinality constraints, but we will see shortly that for four elements this is no longer the case. Thus the variety of degenerate Girard algebras is strictly larger than the variety $\vv{BKL}$ investigated in \cite{AglianoMarcos2022}.
\end{remark}

\subsubsection{Four-element algebras}
Since there is only one 2-element Girard algebra and only one nonintegral 3-element Girard algebra, we turn to 4-element Girard algebras.
It is  obvious that any 4-element Girard algebra has either the 4-element chain or the 4-element diamond as underlying lattice. If we remove the integral ones, we are left with six Girard algebras. Two are the chains
in Figure \ref{figure:1}.

\begin{figure}[h]
\begin{center}
\begin{tikzpicture}
\draw (0,0) -- (0,2);
\draw[fill] (0,0) circle [radius=0.05];
\draw[fill] (0,0.65) circle [radius=0.05];
\draw[fill] (0,1.3) circle [radius=0.05];
\draw[fill] (0,2) circle [radius=0.05];
\node[right]  at (0,0) {\footnotesize $\bot$};
\node[right]  at (0,0.65) {\footnotesize $1$};
\node[right]  at (0,1.3) {\footnotesize $0$};
\node[right]  at (0,2) {\footnotesize $\top = 0^2 =0\top$};
\draw (3,0) -- (3,2);
\draw[fill] (3,0) circle [radius=0.05];
\draw[fill] (3,0.65) circle [radius=0.05];
\draw[fill] (3,1.3) circle [radius=0.05];
\draw[fill] (3,2) circle [radius=0.05];
\node[right]  at (3,0) {\footnotesize $\bot$};
\node[right]  at (3,0.65) {\footnotesize $0=0^2$};
\node[right]  at (3,1.3) {\footnotesize $1$};
\node[right]  at (3,2) {\footnotesize $\top=0\top$};
\node at (0,-.5){\footnotesize $\alg C_4$};
\node at (3,-.5){\footnotesize $\alg C'_4$};
\end{tikzpicture}
\end{center}
\caption{4-element Girard chains}\label{figure:1}
\end{figure}
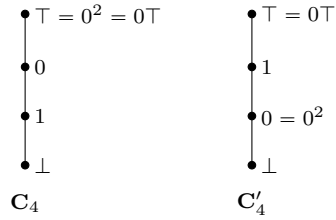

It is easy to check that $\alg C'_4$ is not simple; identifying $0$ and $1$ gives its only proper nontrivial quotient, isomorphic to $\alg K_3$. Since $\alg C'_4$ is $0$-generated, it has no proper subalgebras; hence J\'onsson's Lemma shows that $\VV(\alg C'_4)$ covers $\VV(\alg K_3)$. This cover is not a variety of Kalman lattices.

On the other hand $\alg C_4$ is strictly simple, so $\VV(\alg C_4)$ is an atom in $\Lambda(\mathsf{GA})$.

The two  Girard diamonds in which $0\ne 1$ are in Figure \ref{figure:2}.
It is routine to check that they are both simple and $0$-generated so they both generate atoms.

\begin{figure}[h]
\begin{center}
\begin{tikzpicture}
\draw (0,0) -- (-1,1) -- (0,2) -- (1,1) -- (0,0);
\draw[fill] (0,0) circle [radius=0.05];
\draw[fill] (-1,1) circle [radius=0.05];
\draw[fill] (0,2) circle [radius=0.05];
\draw[fill] (1,1) circle [radius=0.05];
\node[right]  at (0,0) {\footnotesize $\bot$};
\node[right]  at (1,1) {\footnotesize $0$};
\node[left]  at (-1,1) {\footnotesize $1$};
\node[right]  at (0,2) {\footnotesize $\top=0^2=0\top$};
\draw (5,0) -- (4,1) -- (5,2) -- (6,1) -- (5,0);
\draw[fill] (5,0) circle [radius=0.05];
\draw[fill] (4,1) circle [radius=0.05];
\draw[fill] (5,2) circle [radius=0.05];
\draw[fill] (6,1) circle [radius=0.05];
\node[right]  at (5,0) {\footnotesize $\bot$};
\node[right]  at (6,1) {\footnotesize $0$};
\node[left]  at (4,1) {\footnotesize $0^2=1$};
\node[right]  at (5,2) {\footnotesize $\top=0\top$};
\node at (0,-.5){\footnotesize $\alg D_1$};
\node at (5,-.5){\footnotesize $\alg D_2$};
\end{tikzpicture}
\end{center}
\caption{4-element Girard diamonds, where $0\ne1$}\label{figure:2}
\end{figure}

The two degenerate Girard diamonds are in Figure \ref{figure:2deg}.
The first one is the Kalman lattice $\alg K_4$, while the second is not a Kalman lattice. Both are simple and have $\alg K_3$ as their only proper subalgebra; hence J\'onsson's Lemma shows that both generate covers of $\VV(\alg K_3)$.

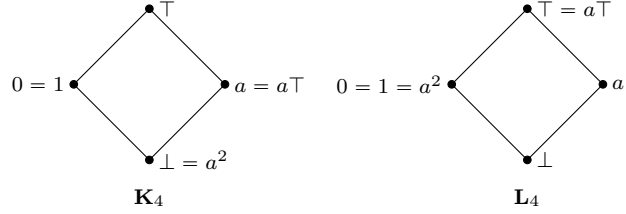
\begin{figure}[h]
\begin{center}
\begin{tikzpicture}
\draw (0,0) -- (-1,1) -- (0,2) -- (1,1) -- (0,0);
\draw[fill] (0,0) circle [radius=0.05];
\draw[fill] (-1,1) circle [radius=0.05];
\draw[fill] (0,2) circle [radius=0.05];
\draw[fill] (1,1) circle [radius=0.05];
\node[right]  at (0,0) {\footnotesize $\bot =a^2$};
\node[right]  at (1,1) {\footnotesize $a=a\top$};
\node[left]  at (-1,1) {\footnotesize $0=1$};
\node[right]  at (0,2) {\footnotesize $\top$};
\draw (5,0) -- (4,1) -- (5,2) -- (6,1) -- (5,0);
\draw[fill] (5,0) circle [radius=0.05];
\draw[fill] (4,1) circle [radius=0.05];
\draw[fill] (5,2) circle [radius=0.05];
\draw[fill] (6,1) circle [radius=0.05];
\node[right]  at (5,0) {\footnotesize $\bot$};
\node[right]  at (6,1) {\footnotesize $a$};
\node[left]  at (4,1) {\footnotesize $0=1 =a^2$};
\node[right]  at (5,2) {\footnotesize $\top=a\top$};
\node at (0,-.5){\footnotesize $\alg K_4$};
\node at (5,-.5){\footnotesize $\alg L_4$};
\end{tikzpicture}
\end{center}
\caption{4-element Girard diamonds, where $0=1$}\label{figure:2deg}
\end{figure}

\subsubsection{Five elements and beyond}

 There are five nonisomorphic lattices with five elements: the $5$-element chain, $\alg M_3$, $\alg N_5$, the $5$-element lattice with one atom and the $5$-element lattice with one coatom (the latter two are sometimes called {\em kites}). Since $\simA$ is an order-reversing involution, the underlying lattice of a Girard algebra is self-dual. This excludes the two kites, so the underlying lattice of a $5$-element Girard algebra is either a chain, $\alg N_5$, or $\alg M_3$. The possible Girard-algebra structures can be checked by hand using the description of residuated lattices up to size $6$ in \cite{GalatosJipsen}, or else through a computer program such as Mace4.

There are exactly 18 nonintegral 5-element Girard algebras, 6 of which are degenerate. The nondegenerate chain-based ones are in Figure \ref{figure:3}. $\alg A_2$ is not simple; it is easy to check that if $\th$ is the congruence generated by $(0,1)$ then $\alg A_2/\th \cong \alg K_3$.   Moreover its only proper subalgebra is $\alg C'_4$ and its only proper nontrivial quotient is $\alg K_3$, so by J\'onsson's Lemma it generates a cover of $\VV(\alg C'_4)$.  $\alg A_{1,1}$ and $\alg A_{1,2}$ are simple and have $\alg C_4$ as their only proper subalgebra, so by J\'onsson's Lemma they both generate covers of $\VV(\alg C_4)$. On the other hand  $\alg C_5$ is simple: its least filter is $\{1,\top\}$, and every larger filter contains $a$, whence $a^2=\bot$ forces the filter to be the whole algebra. It has no proper subalgebras (since $0\top=a$), so it is strictly simple and generates an atom.

\begin{figure}[h]
\begin{center}
\begin{tikzpicture}
\draw (0,0) -- (0,2.7);
\draw[fill] (0,0) circle [radius=0.05];
\draw[fill] (0,0.65) circle [radius=0.05];
\draw[fill] (0,1.3) circle [radius=0.05];
\draw[fill] (0,2) circle [radius=0.05];
\draw[fill] (0,2.7) circle [radius=0.05];
\node[right]  at (0,0) {\footnotesize $\bot$};
\node[right]  at (0,0.65) {\footnotesize $0= 0^2=a^2$};
\node[right]  at (0,1.3) {\footnotesize $a$};
\node[right]  at (0,2) {\footnotesize $1$};
\node[right]  at (0,2.7) {\footnotesize $\top =0\top$};
\draw (3,0) -- (3,2.7);
\draw[fill] (3,0) circle [radius=0.05];
\draw[fill] (3,0.65) circle [radius=0.05];
\draw[fill] (3,1.3) circle [radius=0.05];
\draw[fill] (3,2) circle [radius=0.05];
\draw[fill] (3,2.7) circle [radius=0.05];
\node[right]  at (3,0) {\footnotesize $\bot=0^2=a^2$};
\node[right]  at (3,0.65) {\footnotesize $0$};
\node[right]  at (3,1.3) {\footnotesize $a=0\top$};
\node[right]  at (3,2) {\footnotesize $1$};
\node[right]  at (3,2.7) {\footnotesize $\top$};
\draw (6,0) -- (6,2.7);
\draw[fill] (6,0) circle [radius=0.05];
\draw[fill] (6,0.65) circle [radius=0.05];
\draw[fill] (6,1.3) circle [radius=0.05];
\draw[fill] (6,2) circle [radius=0.05];
\draw[fill] (6,2.7) circle [radius=0.05];
\node[right]  at (6,0) {\footnotesize $\bot$};
\node[right]  at (6,0.65) {\footnotesize $1$};
\node[right]  at (6,1.3) {\footnotesize $a$};
\node[right]  at (6,2) {\footnotesize $0=a^2$};
\node[right]  at (6,2.7) {\footnotesize $\top=0^2=0\top$};
\draw (9,0) -- (9,2.7);
\draw[fill] (9,0) circle [radius=0.05];
\draw[fill] (9,0.65) circle [radius=0.05];
\draw[fill] (9,1.3) circle [radius=0.05];
\draw[fill] (9,2) circle [radius=0.05];
\draw[fill] (9,2.7) circle [radius=0.05];
\node[right]  at (9,0) {\footnotesize $\bot$};
\node[right]  at (9,0.65) {\footnotesize $1$};
\node[right]  at (9,1.3) {\footnotesize $a=a^2$};
\node[right]  at (9,2) {\footnotesize $0$};
\node[right]  at (9,2.7) {\footnotesize $\top=0^2=0\top$};
\node at (0,-.5){\footnotesize $\alg A_2$};
\node at (3,-.5){\footnotesize $\alg C_5$};
\node at (6,-.5){\footnotesize $\alg A_{1,1}$};
\node at (9,-.5){\footnotesize $\alg A_{1,2}$};
\end{tikzpicture}
\end{center}
\caption{5-element Girard chains}\label{figure:3}
\end{figure}
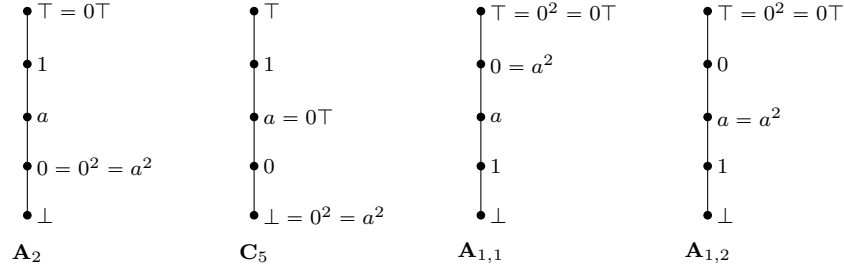

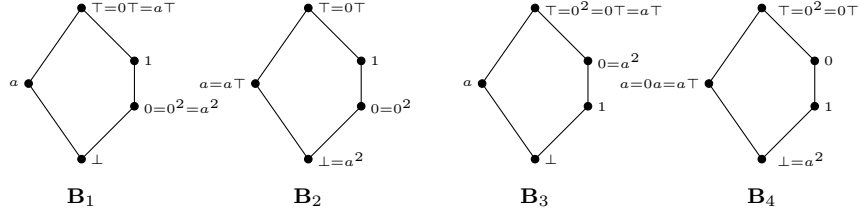
\begin{figure}[h]
\begin{center}
\begin{tikzpicture}
\draw (0,0) -- (-.7,1) -- (0,2) -- (.7,1.3) -- (.7,.7) -- (0,0);
\draw[fill] (0,0) circle [radius=0.05];
\draw[fill] (-.7,1) circle [radius=0.05];
\draw[fill] (0,2) circle [radius=0.05];
\draw[fill] (.7,1.3) circle [radius=0.05];
\draw[fill] (.7,.7) circle [radius=0.05];
\node[right]  at (0,0) {\tiny $\bot$};
\node[left] at (-.7,1){\tiny $a$};
\node[right]  at (0,2) {\tiny \text{$\top$=$0\top$=$a\top$}};
\node[right]  at (.7,1.3) {\tiny $1$};
\node[right]  at (.7,.7) {\tiny \text{$0$=$0^2$=$a^2$}};
\draw (3,0) -- (2.3,1) -- (3,2) -- (3.7,1.3) -- (3.7,.7) -- (3,0);
\draw[fill] (3,0) circle [radius=0.05];
\draw[fill] (2.3,1) circle [radius=0.05];
\draw[fill] (3,2) circle [radius=0.05];
\draw[fill] (3.7,1.3) circle [radius=0.05];
\draw[fill] (3.7,.7) circle [radius=0.05];
\node[right]  at (3,0) {\tiny \text{$\bot$=$a^2$}};
\node[left] at (2.3,1){\tiny \text{$a$=$a\top$}};
\node[right]  at (3,2) {\tiny \text{$\top$=$0\top$}};
\node[right]  at (3.7,1.3) {\tiny $1$};
\node[right]  at (3.7,.7) {\tiny \text{$0$=$0^2$}};
\draw (6,0) -- (5.3,1) -- (6,2) -- (6.7,1.3) -- (6.7,.7) -- (6,0);
\draw[fill] (6,0) circle [radius=0.05];
\draw[fill] (5.3,1) circle [radius=0.05];
\draw[fill] (6,2) circle [radius=0.05];
\draw[fill] (6.7,1.3) circle [radius=0.05];
\draw[fill] (6.7,.7) circle [radius=0.05];
\node[right]  at (6,0) {\tiny $\bot$};
\node[left] at (5.3,1){\tiny  $a$};
\node[right]  at (6,2) {\tiny \text{$\top$=$0^2$=$0\top$=$a\top$}};
\node[right]  at (6.7,1.3) {\tiny \text{$0$=$a^2$}};
\node[right]  at (6.7,.7) {\tiny $1$};
\draw (9,0) -- (8.3,1) -- (9,2) -- (9.7,1.3) -- (9.7,.7) -- (9,0);
\draw[fill] (9,0) circle [radius=0.05];
\draw[fill] (8.3,1) circle [radius=0.05];
\draw[fill] (9,2) circle [radius=0.05];
\draw[fill] (9.7,1.3) circle [radius=0.05];
\draw[fill] (9.7,.7) circle [radius=0.05];
\node[right]  at (9,0) {\tiny \text{ $\bot$=$a^2$}};
\node[left] at (8.3,1){\tiny \text{$a$=$0a$=$a\top$}};
\node[right]  at (9,2) {\tiny \text{ $\top$=$0^2$=$0\top$}};
\node[right]  at (9.7,1.3) {\tiny  $0$};
\node[right]  at (9.7,.7) {\tiny $1$};
\node at (0,-.5){\footnotesize $\alg B_1$};
\node at (3,-.5){\footnotesize $\alg B_2$};
\node at (6,-.5){\footnotesize $\alg B_3$};
\node at (9,-.5){\footnotesize $\alg B_4$};
\end{tikzpicture}
\end{center}
\caption{5-element nonmodular Girard algebras}\label{figure:4}
\end{figure}

There are four 5-element nondegenerate Girard algebras based on $\alg N_5$, displayed in Figure \ref{figure:4}.  It is easy to check that each of them has proper subalgebras: more precisely the $0$-generated subalgebras
of $\alg B_1$ and $\alg B_2$ are isomorphic with $\alg C'_4$ and the $0$-generated subalgebras of $\alg B_3$ and $\alg B_4$ are isomorphic with $\alg C_4$. Both $\alg B_3$ and $\alg B_4$ are simple, with $\alg C_4$ as their only proper subalgebra, so they generate covers of $\VV(\alg C_4)$ which is the only atom below them. On the other hand $\alg B_1$ and $\alg B_2$ are not simple: the congruence generated by $(0,1)$ is nontrivial. In case of $\alg B_1$ it yields an image isomorphic with
$\alg L_4$ and in case of $\alg B_2$ it yields an image isomorphic with  $\alg K_4$.
 However, these are not covers of the individual four-element varieties. In each pair the varieties are incomparable: their generators are nonisomorphic four-element subdirectly irreducible algebras, so J\'onsson's Lemma excludes either containment. Moreover $\alg B_1$ and $\alg B_2$ are subdirectly irreducible, since their proper filters are exactly $\{1,\top\}$ and $\{0,1,\top\}$. Any ultraproduct of factors chosen from either of the two relevant four-element types is isomorphic to one of those types, since the ultrafilter selects one cell of the finite partition by factor type. Hence J\'onsson's Lemma and cardinality give
\[
\VV(\alg C'_4),\ \VV(\alg L_4)
\;<\;\VV(\alg C'_4)\join\VV(\alg L_4)
\;<\;\VV(\alg B_1),
\]
\[
\VV(\alg C'_4),\ \VV(\alg K_4)
\;<\;\VV(\alg C'_4)\join\VV(\alg K_4)
\;<\;\VV(\alg B_2).
\]
By Lemma \ref{lem:relativespectra}, $\VV(\alg K_3)$ is the only atom below either variety.

There are four nondegenerate Girard algebras left; they are all $\alg M_3$ based and displayed in Figure \ref{figure:5}. They are also all simple, since $\alg M_3$ is a simple lattice.
Now it is easily seen that the $0$-generated subalgebra of $\alg D_{1,1}$ and $\alg E_1$ is $\alg D_1$ and the $0$-generated subalgebra of $\alg E_2$ is $\alg D_2$. Since both $\alg D_1$ and $\alg D_2$ are simple they generate the only atoms below the corresponding varieties.  On the other hand $\alg D_5$ is $0$-generated, so it generates an atom
in $\Lambda(\mathsf{GA})$.

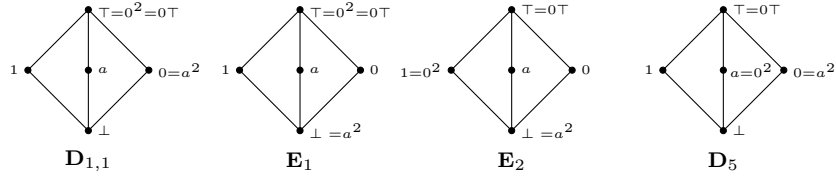
\begin{figure}[h]
\begin{center}
\begin{tikzpicture}[scale=0.8]
\draw (0,0) -- (-1,1) -- (0,2) -- (1,1) -- (0,0);
\draw (0,0) -- (0,2);
\draw[fill] (0,0) circle [radius=0.05];
\draw[fill] (-1,1) circle [radius=0.05];
\draw[fill] (0,1) circle [radius=0.05];
\draw[fill] (0,2) circle [radius=0.05];
\draw[fill] (1,1) circle [radius=0.05];
\node[right]  at (0,0) {\tiny  $\bot$};
\node[right]  at (1,1) {\tiny \text{$0$=$a^2$}};
\node[right]  at (0,1) {\tiny $a$};
\node[left]  at (-1,1) {\tiny $1$};
\node[right]  at (0,2) {\tiny \text{$\top$=$0^2$=$0\top$}};
\draw (3.5,0) -- (2.5,1) -- (3.5,2) -- (4.5,1) -- (3.5,0);
\draw (3.5,0) -- (3.5,2);
\draw[fill] (3.5,0) circle [radius=0.05];
\draw[fill] (2.5,1) circle [radius=0.05];
\draw[fill] (3.5,1) circle [radius=0.05];
\draw[fill] (3.5,2) circle [radius=0.05];
\draw[fill] (4.5,1) circle [radius=0.05];
\node[right]  at (3.5,0) {\tiny \text{$\bot$ =$a^2$}};
\node[right]  at (4.5,1) {\tiny $0$};
\node[right]  at (3.5,1) {\tiny $a$};
\node[left]  at (2.5,1) {\tiny $1$};
\node[right]  at (3.5,2) {\tiny \text{$\top$=$0^2$=$0\top$}};
\draw (7,0) -- (6,1) -- (7,2) -- (8,1) -- (7,0);
\draw (7,0) -- (7,2);
\draw[fill] (7,0) circle [radius=0.05];
\draw[fill] (6,1) circle [radius=0.05];
\draw[fill] (7,1) circle [radius=0.05];
\draw[fill] (7,2) circle [radius=0.05];
\draw[fill] (8,1) circle [radius=0.05];
\node[right]  at (7,0) {\tiny \text{$\bot$ =$a^2 $}};
\node[right]  at (8,1) {\tiny $0$};
\node[right]  at (7,1) {\tiny $a$};
\node[left]  at (6,1) {\tiny \text{$1$=$0^2$}};
\node[right]  at (7,2) {\tiny \text{$\top$=$0\top$}};
\draw (10.5,0) -- (9.5,1) -- (10.5,2) -- (11.5,1) -- (10.5,0);
\draw (10.5,0) -- (10.5,2);
\draw[fill] (10.5,0) circle [radius=0.05];
\draw[fill] (9.5,1) circle [radius=0.05];
\draw[fill] (10.5,1) circle [radius=0.05];
\draw[fill] (10.5,2) circle [radius=0.05];
\draw[fill] (11.5,1) circle [radius=0.05];
\node[right]  at (10.5,0) {\tiny $\bot $};
\node[right]  at (11.5,1) {\tiny \text{$0$=$a^2$}};
\node[right]  at (10.45,1) {\tiny \text{$a$=$0^2$}};
\node[left]  at (9.5,1) {\tiny $1$};
\node[right]  at (10.5,2) {\tiny \text{$\top$=$0\top$}};
\node at (0,-.5){\footnotesize $\alg D_{1,1}$};
\node at (3.5,-.5){\footnotesize $\alg E_1$};
\node at (7,-.5){\footnotesize $\alg E_2$};
\node at (10.5,-.5){\footnotesize $\alg D_5$};
\end{tikzpicture}
\end{center}
\caption{5-element modular Girard algebras}\label{figure:5}
\end{figure}

To complete the classification we mention that one can show that, of the 6 degenerate 5-element Girard algebras, one is a chain, two are $\alg N_5$ based and three are $\alg M_3$ based. The reader can have fun in discovering them.

There are many nonintegral Girard algebras with six elements and they can in principle be classified using the same techniques as above. However, the full classification quickly becomes lengthy and is not needed here. Nevertheless we want to present at least a $6$-element Girard algebra which generates an atom. Its underlying lattice is the product of the $2$-element and $3$-element chains:
\begin{figure}[h]
\begin{center}
\begin{tikzpicture}
\draw (0,0) -- (-1,1) -- (0,2)-- (1,3) -- (2,2) -- (1,1) -- (0,0);
\draw (1,1) -- (0,2);
\draw[fill] (0,0) circle [radius=0.05];
\draw[fill] (-1,1) circle [radius=0.05];
\draw[fill] (0,2) circle [radius=0.05];
\draw[fill] (1,1) circle [radius=0.05];
\draw[fill] (1,3) circle [radius=0.05];
\draw[fill] (2,2) circle [radius=0.05];
\node[right]  at (0,0) {\footnotesize $\bot$};
\node[right]  at (1,1) {\footnotesize $a$};
\node[left]  at (-1,1) {\footnotesize $0$};
\node[left]  at (0,2) {\footnotesize $b$};
\node[right]  at (2,2) {\footnotesize $1$};
\node[right]  at (1,3) {\footnotesize $\top$};
\end{tikzpicture}
\end{center}
\caption{}\label{figure:2x3}
\end{figure}

 Its multiplication table is:

\begin{center}
\begin{tabular}{c|cccccc}
  $\cdot$ & $\bot$ & $0$ & $a$ & $b$ & $1$ & $\top$ \\
  \hline
  $\bot$ & $\bot$ & $\bot$ & $\bot$ & $\bot$ & $\bot$ & $\bot$ \\
  $0$ & $\bot$ & $a$ & $\bot$ & $a$ & 0 & $b$ \\
  $a$ & $\bot$ & $\bot$ & $\bot$ & $\bot$ & $a$ & $a$ \\
  $b$ & $\bot$ & $a$ & $\bot$ & $a$ & $b$ & $b$ \\
  $1$ & $\bot$& $0$ & $a$ & $b$ & $1$ & $\top$ \\
  $\top$ & $\bot$ & $b$ & $a$ & $b$ & $\top$ & $\top$ \\
\end{tabular}.
\end{center}

A routine verification shows that this is a Girard algebra and that it is simple and $0$-generated. Hence it generates an atom in $\Lambda(\mathsf{GA})$.

So far all the atoms we have presented are based on modular lattices. Now we will present an example based on a nonmodular lattice with 7 elements, which has been called $\alg G$ in \cite{JonssonRival1979}.
The lattice structure is

\begin{figure}[h]
\begin{center}
\begin{tikzpicture}
\draw (0,0) -- (-1,1) -- (0,2)-- (1,3) -- (2,2) -- (1,1) -- (0,0);
 \draw (1,1) -- (.5,1.5) -- (0,2);
\draw[fill] (0,0) circle [radius=0.05];
\draw[fill] (-1,1) circle [radius=0.05];
\draw[fill] (.5,1.5) circle [radius=0.05];
\draw[fill] (0,2) circle [radius=0.05];
\draw[fill] (1,1) circle [radius=0.05];
\draw[fill] (1,3) circle [radius=0.05];
\draw[fill] (2,2) circle [radius=0.05];
\node[right]  at (0,0) {\footnotesize $\bot$};
\node[right]  at (1,1) {\footnotesize $a$};
\node[left]  at (.5,1.5) {\footnotesize $c$};
\node[left]  at (-1,1) {\footnotesize $0$};
\node[left]  at (0,2) {\footnotesize $b$};
\node[right]  at (2,2) {\footnotesize $1$};
\node[right]  at (1,3) {\footnotesize $\top$};
\end{tikzpicture}
\end{center}
\caption{The lattice $\alg G$}\label{figure:G}
\end{figure}
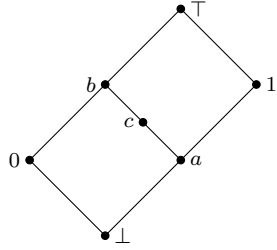
and it is clearly nonmodular (it contains an $\alg N_5$).
The multiplication is
\[
\begin{array}{c|ccccccc}
	\cdot&\bot&a&c&0&b&1&\top\\ \hline
	\bot&\bot&\bot&\bot&\bot&\bot&\bot&\bot\\
	a&\bot&\bot&\bot&\bot&\bot&a&a\\
	c&\bot&\bot&\bot&a&a&c&c\\
	0&\bot&\bot&a&c&c&0&b\\
	b&\bot&\bot&a&c&c&b&b\\
	1&\bot&a&c&0&b&1&\top\\
	\top&\bot&a&c&b&b&\top&\top
\end{array}.
\]
 A direct check shows that multiplication is associative, that $\bot$ is absorbing, and that multiplication preserves joins in each coordinate; since the lattice is finite, this is enough to guarantee residuation. The induced negation $\simA x=x\imp0$ is
\begin{center}
\begin{tabular}{c|ccccccc}
   $\simA$ & $\bot$ & $0$ & $1$ & $a$  & $b$ & $c$ & $\top$ \\
  \hline
    & $\top$ & $1$ & $0$ & $b$ & $a$ & $c$ & $\bot$ \\
\end{tabular}
\end{center}
so it is involutive and $\alg G$ is a Girard algebra. Moreover
\[
c=0^2,\qquad a=0^3,\qquad b=0\top,
\]
hence it is $0$-generated. It is also simple: every filter properly containing
$\{1,\top\}$ contains one of $a,c,0,b$, and
\[
a^2=c^2=\bot,\qquad 0^2=b^2=c,
\]
so every such filter contains $\bot$.

Thus $\alg G$ is simple and $0$-generated with nonmodular lattice reduct,
and therefore $\VV(\alg G)$ is a nonmodular atom of
$\Lambda(\mathsf{GA})$.

\subsubsection{A countable chain above $\alg D_1$}  Let  $\alg D_{1,n}$, $n \ge 0$ be the modular lattice of height 3 with $n+2$ coatoms
(also denoted $\alg M_{n+2}$), where the coatoms are $0,1,\vuc an$ and the largest and smallest element are $\top$ and $\bot$. We define a binary commutative operation in the following way:
\begin{align*}
&x\cdot \bot =\bot\quad\text{for all $x$}\\
&x\cdot 1= x \quad\text{for all $x$}\\
&x\cdot y =\top\quad\text{if $x \ne y$ and $x,y \notin \{\bot,1\}$}\\
&x\cdot x =0 \quad \text{if $x=a_i$, $i \le n$}\\
&0\cdot 0 = 0 \cdot \top = \top \cdot \top = \top.
\end{align*}

\begin{lemma} For each $n\ge 0$ the operation $\cdot$ is residuated and induces a structure of a Girard algebra on $\alg D_{1,n}$; moreover
\begin{enumerate}
\item $\alg D_{1,0} \cong \alg D_1$;
\item each $\alg D_{1,n}$ is simple and, up to isomorphism, its only subalgebras are the $\alg D_{1,k}$ with $k \le n$.
\end{enumerate}
\end{lemma}
\begin{proof}

It is straightforward to check that $\cdot$ is associative and  preserves binary joins and $\bot$ so each $\alg D_{1,n}$ is a bounded $\mathsf{FL}_{e}$-algebra. It is also easy to see that $\simA 1= 0$, $\simA 0 =1$, $\simA \top = \bot$ and $\simA \bot =\top$.

Let now $c \in \{\vuc an\}$; then
$$
\simA c = c \imp 0 = \bigvee\{z: z\cdot c \in \{0,\bot\}\} = c.
$$
This implies that $0$ is involutive and so $\alg D_{1,n}$ is a Girard algebra.

That $\alg D_{1,0} \cong \alg D_1$ is obvious so it is simple.
For $n>0$ $\alg D_{1,n}$  is simple, since it is already simple as a lattice.  That, up to isomorphism, its only subalgebras are the $\alg D_{1,k}$ with $k \le n$ follows from a straightforward calculation.
\end{proof}

So by J\'onsson's Lemma  $\VV(\alg D_{1,n})$ is covered by $\VV(\alg D_{1,n+1})$ for $n \ge 0$. This means that there is a countable chain
$$
\VV(\alg D_1) \prec \VV(\alg D_{1,1}) \prec \dots \prec \VV(\alg D_{1,n}) \prec \dots
$$
above $\alg D_1$. Here is a picture of $\alg D_{1,2}$:

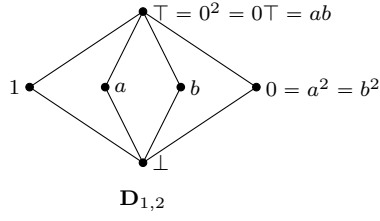
\begin{figure}[h]
\begin{center}
\begin{tikzpicture}
\draw (0,0) -- (-1.5,1) -- (0,2) -- (1.5,1) -- (0,0);
\draw (0,0) -- (-.5,1) -- (0,2);
\draw (0,0) -- (.5,1) -- (0,2);
\draw[fill] (0,0) circle [radius=0.05];
\draw[fill] (-1.5,1) circle [radius=0.05];
\draw[fill] (-.5,1) circle [radius=0.05];
\draw[fill] (.5,1) circle [radius=0.05];
\draw[fill] (1.5,1) circle [radius=0.05];
\draw[fill] (0,2) circle [radius=0.05];
\node[right]  at (0,0) {\footnotesize $\bot$};
\node[right]  at (1.5,1) {\footnotesize $0=a^2=b^2$};
\node[right]  at (-.5,1) {\footnotesize $a$};
\node[right]  at (.5,1) {\footnotesize $b$};
\node[left]  at (-1.5,1) {\footnotesize $1$};
\node[right]  at (0,2) {\footnotesize $\top=0^2=0\top=ab$};
\node at (0,-.5){\footnotesize $\alg D_{1,2}$};
\end{tikzpicture}
\end{center}
\caption{The algebra $\alg D_{1,2}$.}\label{figure:6}
\end{figure}

\subsubsection{A countable family of totally ordered atoms}

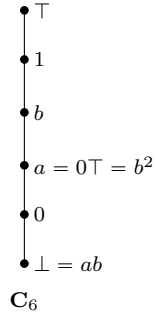
\begin{figure}[htbp]
\begin{center}
\begin{tikzpicture}
\draw (0,0) -- (0,3.35);
\draw[fill] (0,0) circle [radius=0.05];
\draw[fill] (0,0.65) circle [radius=0.05];
\draw[fill] (0,1.3) circle [radius=0.05];
\draw[fill] (0,2) circle [radius=0.05];
\draw[fill] (0,2.7) circle [radius=0.05];
\draw[fill] (0,3.35) circle [radius=0.05];
\node[right]  at (0,0) {\footnotesize $\bot=ab$};
\node[right]  at (0,0.65) {\footnotesize $0$};
\node[right]  at (0,1.3) {\footnotesize $a=0\top=b^2$};
\node[right]  at (0,2) {\footnotesize $b$};
\node[right]  at (0,2.7) {\footnotesize $1$};
\node[right]  at (0,3.35) {\footnotesize $\top$};
\node at (0,-.5){\footnotesize $\alg C_6$};
\end{tikzpicture}
\end{center}
\caption{The Girard chain $\alg C_6$.}\label{figure:7}
\end{figure}

Consider the 6-element Girard chain in Figure \ref{figure:7};
its complete multiplication table is:
\begin{center}
\begin{tabular}{c|cccccc}
  $\cdot$ & $\bot$ & $0$ & $a$ & $b$ & $1$ & $\top$ \\
  \hline
  $\bot$ & $\bot$ & $\bot$ & $\bot$ & $\bot$ & $\bot$ & $\bot$ \\
  $0$ & $\bot$ & $\bot$ & $\bot$ & $\bot$ & 0 & $a$ \\
  $a$ & $\bot$ & $\bot$ & $\bot$ & $\bot$ & $a$ & $a$ \\
  $b$ & $\bot$ & $\bot$ & $\bot$ & $a$ & $b$ & $b$ \\
  $1$ & $\bot$& $0$ & $a$ & $b$ & $1$ & $\top$ \\
  $\top$ & $\bot$ & $a$ & $a$ & $b$ & $\top$ & $\top$ \\
\end{tabular}
\end{center}
Again a straightforward computation shows that $\alg C_6$ is simple and $0$-generated, so it generates an atom.
Now the comparison of $\alg C_5$ and $\alg C_6$ suggests a pattern.  A careful analysis gives a recipe for constructing a chain of any length with the same behavior. The chain  $\alg C_{n+4}$, for $n\ge1$, has length $n+4$ and the ordering is
$$
\bot < 0 <a_1 < a_2 < \dots <a_n <1 < \top.
$$

 Besides $\bot x=\bot$, $1x=x$, and $\top^2=\top$, the nonobvious products are:
\begin{align*}
&0^2=\bot\\
&0a_i = \bot\\
&0\top= a_1\\
&a_i\top=a_i\\
&a_ia_j =\left\{
         \begin{array}{ll}
           \bot, & \hbox{if $i+j \le n+1$;} \\
            a_{i+j-n-1}, & \hbox{otherwise.}
         \end{array}
       \right.
\end{align*}
It is easy to check that  for $n=1,2$ these conditions give $\alg C_5$ and $\alg C_6$; $\alg C_4$ is the separate example above.
Since the multiplication clearly respects the ordering, to show that $\alg C_{n+4}$ is a bounded $\mathsf{FL}_{e}$-algebra it is enough to check that the multiplication is associative.

The cases involving $\bot,0,1,\top$ follow directly from the displayed products. For the remaining cases put $L=n+1$ and encode $\bot$ by the index $0$. Multiplication of indices is $i*j=\max(0,i+j-L)$, and for $1\le i,j,k\le n$,
\[
(i*j)*k=\max(0,i+j+k-2L)=i*(j*k).
\]
Thus multiplication is associative.

Next we show that $0$ is involutive. To do so it is enough to observe that $\simA \top=\bot$, $\simA \bot =\top$, $\simA 0= 1$, $\simA 1=0$ and for all
 $k\le n$, $\simA a_k = a_{(n-k)+1}$.

To prove that $\alg C_{n+4}$ is simple we show that every nonminimal filter of $ \alg C_{n+4}$ is the entire algebra. First it is easy to see, since it is true in any
residuated lattice, that the minimal filter is $\{1,\top\}$. If  $F$  is nonminimal, then $a_i \in F$ for some $i$ and hence $a_ia_i \in F$.
Now if $2i \le n+1$ (i.e. $i \le (n+1)/2$) $a_ia_i =\bot$; hence $\bot \in F$ and $F$ is the entire algebra.  Otherwise
$a_ia_i = a_{2i-n-1}$ and $2i-n-1 < i$ which implies $a_{2i-n-1} <a_i$. Let $u=2i-n-1$; if $a_ua_u \ne \bot$,  repeat the argument and obtain a $v$ such that
$a_v < a_u$ and $a_v \in F$. After a finite number of steps we must find a $v'$ such that $a_{v'} \in F$ and $ 2v' \le n+1$. This as before will force $\bot \in F$ and hence again $F$ is the entire algebra.

Finally $\alg C_{n+4}$ has no proper subalgebras. The constants give $a_1=0\top$, and involution gives $a_n=\simA a_1$. For $2\le j\le n$, $a_na_j=a_{j-1}$, so all the remaining elements are generated as well; the case $n=1$ is already settled.

So for any chain $\alg C_{n+4}$, $\VV( \alg C_{n+4})$ is an atom in $\Lambda(\mathsf{GA})$.

\subsubsection{An infinite family of modular atoms}\label{subsec:modular atoms}

Now we present a construction that produces an infinite family of modular atoms (and one that is not finitely generated).
Let $\alg G$ be an abelian group, written multiplicatively with identity $e$,
and let $c\in G$. Let $\mathbf M(\alg G,c)$ be the algebra with universe
$\{\bot\}\cup G\cup\{\top\}$, where $\bot$ and $\top$ are new elements.
We introduce a partial ordering on $\alg M(\alg G,c)$ in the following way:
\begin{enumerate}
\itemb the elements of $G$ form an antichain;
\itemb  $\bot<g<\top$ for every $g\in G$.
\end{enumerate}
Moreover we define:
\[
\bot\cdot x=x\cdot\bot=\bot,\qquad
\top\cdot x=x\cdot\top=\top\quad(x\ne\bot),\qquad
g\cdot h=gh\quad(g,h\in G),
\]
\[
\begin{aligned}
\bot\to y&=\top, & x\to\top&=\top,\\
\top\to y&=\bot\quad(y\ne\top), &
g\to\bot&=\bot,\\
g\to h&=g^{-1}h\quad(g,h\in G).
\end{aligned}
\]
and we add the constants $1:=e$, $0:=c$.

\begin{proposition}\label{prop:M}
Let $\alg G$ be an abelian group and $c\in G$. Then:
\begin{enumerate}
\item[(a)] $\mathbf M(\alg G,c)$ is a nontrivial bounded Girard algebra, with
$\simA g=g^{-1}c$, $\simA\bot=\top$, $\simA\top=\bot$ and least element
$\bot=\simA\top$;
\item[(b)] the subalgebras of $\mathbf M(\alg G,c)$ are exactly the algebras
$\mathbf M(\alg H,c)$ where  $H\le G$ and $c\in H$;
\item[(c)] $\mathbf M(\alg G,c)$ is $0$-generated if and only if $\alg G$ is generated by $c$ as a group.
\end{enumerate}
\end{proposition}

\begin{proof}
(a) The order is a lattice order of height three: for distinct
$g,h\in G$ one has $g\wedge h=\bot$ and $g\vee h=\top$. Multiplication is
commutative because $\alg G$ is abelian and the clauses for $\bot,\top$ are
symmetric; $e$ is a unit since $e\cdot g=g$, $e\cdot\bot=\bot$ and
$e\cdot\top=\top$. It is associative: any product one of whose factors is
$\bot$ equals $\bot$ under every bracketing; any product with no factor
$\bot$ and at least one factor $\top$ equals $\top$ under every
bracketing. Finally, on $G$ this is the group operation.

Residuation, $x\cdot z\le y\iff z\le x\to y$, is checked by cases on
$x$ and $y$. If $x=\bot$, both sides hold for every $z$, matching
$\bot\to y=\top$. If $y=\top$, both sides hold for every $z$, matching
$x\to\top=\top$. If $x=\top$ and $y\ne\top$, then $\top\cdot z\le y$ holds
exactly for $z=\bot$, matching $\top\to y=\bot$. If $x=g\in G$ and
$y=\bot$, then $g\cdot z\in G\cup\{\bot,\top\}$ is $\bot$ exactly for
$z=\bot$, matching $g\to\bot=\bot$. Finally let $x=g$ and $y=h$ with
$g,h\in G$: for $z=\bot$ both sides hold; for $z=\top$ we have
$g\cdot\top=\top\not\le h$ and $\top\not\le g^{-1}h$, so both fail; and
for $z=k\in G$, $gk\le h$ means $gk=h$, i.e.\ $k=g^{-1}h$, which is
exactly $k\le g^{-1}h$ in the antichain. So
$\mathbf M(G,c)$ is a commutative residuated lattice.

For involutivity, $\simA g=g\to c=g^{-1}c$, $\simA\bot=\bot\to c=\top$ and
$\simA\top=\top\to c=\bot$ since $c\ne\top$; hence
$\simA\simA g=(g^{-1}c)^{-1}c=c^{-1}gc=g$ using commutativity, and
$\simA\simA$ fixes $\bot$ and $\top$. Finally $x\to\top=\top\ge e=1$, so
$\top$ satisfies the axiom for the bounded constant, and
$\simA\top=\bot$ is the least element. The algebra is nontrivial since
$\bot\ne\top$.

(b) A subuniverse contains the constants $e,c,\top$ and
$\bot=\simA\top$, and its intersection $H$ with $G$ is closed under
$g\cdot h=gh$ and $g\to h=g^{-1}h$, hence is a subgroup containing $c$.
Conversely for any such $H$ the set $\{\bot\}\cup H\cup\{\top\}$ is closed
under all operations, and the induced algebra is $\mathbf M(\alg H,c)$.

(c) By (b) the $0$-generated subalgebra of $\alg M(\alg G,c)$  is
$\mathbf M(\langle c\rangle,c)$, since $\langle c\rangle$ is the smallest
subgroup containing $c$. This is all of $\mathbf M(\alg G,c)$ exactly when
$G=\langle c\rangle$.
\end{proof}

First we observe that $\alg M(\alg G,c)$ is simple for every $\alg G,c$. Its least filter is $\{e,\top\}$, and any larger filter contains either $\bot$ or a group element $g\ne e$, in which case it also contains $g\meet e=\bot$.

Consequently Proposition \ref{prop:M}(c) and Lemma \ref{lem:relativespectra} give, for every abelian group $\alg G$ and $c\in G$,
\[
\op{At}(\VV(\mathbf M(\alg G,c)))
=\{\VV(\mathbf M(\langle c\rangle,c))\}.
\]
Let $\alg G_n =\langle c_n\rangle$ be the cyclic group of order $n\ge1$, and let $\alg G_\infty=\langle c\rangle\cong\mathbb Z$. The algebras $\mathbf M(\alg G_n,c_n)$ and $\mathbf M(\alg G_\infty,c)$ are simple and $0$-generated, so they generate pairwise distinct atoms by Theorem \ref{thm:atomchar}. This gives an infinite family of finite modular atoms and one atom not generated by any finite algebra: otherwise Lemma \ref{lem:freezero} would identify the $0$-generated core of a finite generator with the infinite algebra $\mathbf M(\alg G_\infty,c)$.
We can now make the infinite-join phenomenon of
Section~\ref{sec:atomsupport} explicit. Write
\[
\begin{aligned}
\alg S_n&=\mathbf M(\alg G_n,c_n),& \vv A_n&=\VV(\alg S_n)\quad(n\ge1),\\
\alg S_\infty&=\mathbf M(\alg G_\infty,c),&
\vv A_\infty&=\VV(\alg S_\infty),
\end{aligned}
\]
and let $M_n,M_\infty\in X$ be the corresponding maximal filters.

\begin{lemma}\label{lem:cyclicterms}
For any closed terms $s,t$, the statements
$\alg S_n\models s\app t$ and $\alg S_\infty\models s\app t$
have the same truth value for all sufficiently large $n$.
\end{lemma}
\begin{proof}
Every closed term $t$ has in $\alg S_\infty$ one of the values
$\bot$, $\top$, or $c^k$ for some integer $k$. We claim that, for all
sufficiently large $n$, its value in $\alg S_n$ is respectively
$\bot$, $\top$, or $c_n^k$. This follows by induction on terms.
It holds for the constants. Products and residuals of group elements
add and subtract their exponents, respectively; all cases involving a
bound follow from the displayed operation tables. For meet and join of
group elements the only distinction is whether the two elements are
equal. For fixed integers $k,l$,
\[
c_n^k=c_n^l\quad\Longleftrightarrow\quad n\mid(k-l),
\]
so this agrees with $k=l$ whenever $n>|k-l|$. Thus each induction step
requires only increasing a finite bound on $n$. Applying the claim to
$s,t$ and, if necessary, increasing the bound once more to distinguish
their exponents proves the lemma.
\end{proof}

\begin{theorem}\label{thm:cycliclimit}
The subspace $Y=\{M_n:n\ge1\}\cup\{M_\infty\}$ of $X$ is homeomorphic
to $\omega+1$ with its order topology: each $M_n$ is isolated in $Y$,
and the neighbourhoods of $M_\infty$ in $Y$ are exactly the sets
containing $M_\infty$ and all but finitely many $M_n$.
For every $S\sse\mathbb N_{\ge1}$,
\begin{equation}\label{eq:cyclicjoin}
\vv A_\infty\le\bigvee_{n\in S}\vv A_n
\quad\Longleftrightarrow\quad S\text{ is infinite}.
\end{equation}
Consequently $\vv A_\infty$ is a noncompact atom of
$\Lambda(\mathsf{GA})$ and is not completely join-prime.
\end{theorem}
\begin{proof}
For closed terms $s,t$ put
\[
U_{s,t}=\{M\in X:(s^{\alg F},t^{\alg F})\notin\th_M\}.
\]
This is open: its complement is  $T(D)$, where $D$ corresponds to the congruence generated by $(s^{\alg F},t^{\alg F})$. These opens form a base, since $D\nsubseteq M$ is equivalent to $\th_D\nsubseteq\th_M$, and every pair of elements of $F$ is represented by closed terms. Lemma~\ref{lem:cyclicterms} therefore shows that every open neighbourhood of $M_\infty$ contains all sufficiently large $M_n$.

To isolate the finite points, for $r\ge1$ consider the closed term
\[
q_r=\bigvee_{k=1}^r(1\wedge0^k).
\]
In $\alg S_n$ it equals $1$ if $n\le r$, and $\bot$ otherwise;
in $\alg S_\infty$ it equals $\bot$. Indeed, $1\wedge c_n^k$ is
$1$ exactly when $n\mid k$, and is $\bot$ otherwise. Hence
$U_{q_r,\bot}\cap Y=\{M_1,\ldots,M_r\}$. Removing the other finitely
many points, which are closed, shows that each $M_r$ is isolated in $Y$.
The $T_1$ property also makes every cofinite subset of $Y$ open.
Together with convergence, this proves the asserted description of the
topology and the homeomorphism with $\omega+1$.

If $S$ is infinite, every neighborhood of $M_\infty$ meets
$\{M_n:n\in S\}$, so Theorem~\ref{thm:atomsupport} gives the
containment in~\eqref{eq:cyclicjoin}. If $S$ is finite, that set is
closed in $X$ and does not contain $M_\infty$, so the same theorem
excludes the containment. Finally, compactness would force
$\vv A_\infty\le\bigvee_{n\ge1}\vv A_n$ to hold already for a finite
subjoin, contrary to~\eqref{eq:cyclicjoin}. Complete join-primality would
force $\vv A_\infty\le\vv A_n$ for some $n$, contrary to the
distinctness of these atoms.
\end{proof}

\subsubsection{A real interval of atoms}\label{subsec:realinterval}

The convergent sequence in Theorem~\ref{thm:cycliclimit} is part of a much larger phenomenon: the atom spectrum contains a closed copy of the real interval $[0,1]$.  We consider the {\em standard $\mathsf{MV}$-algebra} $[0,1]_{\mathsf{MV}}$, a.k.a. the {\em {\L}ukasiewicz t-norm algebra}. It is an $\mathsf{FL}_{ew}$-algebra  whose  universe is the
real interval $[0,1]$ with the usual ordering and the operations are
\[
xy=\max(0,x+y-1),\qquad x\imp y=\min(1,1-x+y),
\].

We apply a classical twist construction on $[0,1]_\mathsf{MV}$;
 for $r\in[0,1]$, we let
\[
T_r=\{(a,b)\in[0,1]^2:ab\le r\}
=\{(a,b):a+b\le1+r\},
\]
ordered by $(a,b)\le(c,d)$ if and only if $a\le c$ and $b\ge d$. If we define
\[
\begin{split}
	(a,b)(c,d)&=\bigl(a c,(a\imp d)\meet(c\imp b)\bigr),\\
	(a,b)\imp(c,d)&=\bigl((a\imp c)\meet(d\imp b),ad\bigr),
\end{split}
\]
and put $\bot=(0,1)$, $\top=(1,0)$, $1_r=(1,r)$ and $0_r=(r,1)$ then:

\begin{proposition} For each $r \in [0,1]$
\[
\alg T_r=\langle T_r,\join,\meet,\cdot,\imp,0_r,1_r,\top,\bot\rangle
\]
is a simple  distributive Girard algebra.
\end{proposition}
\begin{proof} That $\alg T_r$ with the operations so defined is a Girard algebra is a straightforward consequence of general results on twist
structures (see for instance Theorem 5 in \cite{BusanicheGalatosMarcos2022}). Moreover its underlying lattice is a sublattice of the direct product of two chains, so it is distributive.

The smallest filter of $\alg T_r$ is
$$
T^+_r = \{(a,b):  (a,b) \ge (1,r)\} = \{(1,b): b \le r\}.
$$
Hence any strictly larger filter $F$ must contain a pair $(a,b)$ where  $a<1$.
Now we  observe that in $\alg T_r$
$$
(a,b)^{n+1} = (a^{n+1},a^n \imp b)
$$
and that for any $a  < 1$, there is an $\bar n$ such that $a^n =0$ for all $n \ge \bar n$ (the power is computed using the product in $[0,1]_\mathsf{MV}$).

So
$$
\bot = (0,1) = (a^{\bar n+1},a^{\bar n} \imp b) = (a,b)^{\bar n+1} \in F.
$$
Hence $F = T_r$ and $\alg T_r$ is simple.
\end{proof}

Let now $\alg R_r$ be the $0$-generated subalgebra of $\alg T_r$; it is well-known that commutative residuated lattices (hence a fortiori Girard algebras)
have the {\em congruence extension property}. So a subalgebra of a simple algebra must be simple as well. Hence $\alg R_r$ is simple
and $0$-generated, so $\vv U_r =\VV(\alg R_r)$ is an atom in $\Lambda(\mathsf{GA})$ (by Theorem~\ref{thm:atomchar}). We denote by $P_r$ the filter in the atom spectrum
$X$ such that $\alg R_r = \alg F/P_r$.

\begin{theorem}\label{thm:realinterval} The map $\a:r\mapsto P_r$ identifies $[0,1]$ homeomorphically with a closed subspace of the atom spectrum $X$. Moreover for every $E\sse[0,1]$,
		\[
		\op{At}\!\left(\bigvee_{r\in E}\vv U_r\right)
		=\{\vv U_r:r\in\overline E\},
		\]
		where $\overline E$ is the ordinary closure in $[0,1]$.
\end{theorem}
\begin{proof} To prove that the map is injective it is enough to show that if $r \ne s$, then $\vv U_r \ne \vv U_s$.

	Put $\delta(t)=(t,1-t)$. Then $\delta(s)\delta(t)=\delta(st)$ and $\simA\delta(t)=\delta(1-t)$. If $x\oplus y=\simA(\simA x\cdot\simA y)$, then $\delta(s)\oplus\delta(t)=\delta(\min(1,s+t))$; in particular the closed term $u=0\top$ has value $\delta(r)$ in $\alg R_r$.
	
Define closed terms by
	\[
	p_{0,0}=u,\qquad
	p_{k+1,2j}=p_{k,j}\oplus p_{k,j},\qquad
	p_{k+1,2j+1}=p_{k,j}p_{k,j}.
	\]
	A direct induction gives
	\[
	p_{k,j}^{\alg R_r}
	=\delta\!\left(\min(1,\max(0,2^kr-j))\right).
	\tag{*}
	\]
Now
\begin{align*}
&p_{k,j}^{\alg R_r} =\bot\qquad\text{if and only if}\\
&\min(1,\max(0,2^kr-j)) = 0\qquad\text{if and only if}\\
&\max(0,2^kr-j)  =0 \qquad\text{if and only if}\\
&r \le \frac{j}{2^k}.
\end{align*}
Similarly $p_{k,j}^{\alg R_r} =\top$ if and only if $r \ge \frac{j+1}{2^k}$.  Now suppose that $r \ne s$, say $r < s$; since dyadic numbers are dense we can pick $j,k$ such that $r < \frac{j}{2^k}<s$. Then
$\vv U_r \vDash p_{k,j}\app\bot$ while $\vv U_s \not\vDash p_{k,j} \app\bot$. This shows that $\vv U_r \ne \vv U_s$ and $\a$ is injective.

Now consider  $Y =\a([0,1])$ as a subspace of $X$; then a basis for the subspace topology on $Y$ is just $\{Y \cap U_{s,t}: s,t\ \text{closed terms}\}$, where $U_{s,t}$ are the basis for $X$ as in Theorem~\ref{thm:cycliclimit}.  Every closed term has in $\alg R_r$ two coordinates depending continuously on $r$, so the basic open sets $U_{s,t}$  have open inverse images under $\a$. So the same holds for the corresponding basis for $Y$ and
$\a$ is continuous.

Conversely, by $(*)$ (and the following argument)
	\[
	P_r\in U_{p_{k,j},\bot}\iff r>j/2^k,\qquad
	P_r\in U_{p_{k,j},\top}\iff r<(j+1)/2^k.
	\]
 Since the dyadic rays $[0,\frac{j}{2^k})$ and $(\frac{j}{2^k},1]$ form a subbasis for the subspace topology on $[0,1]$, it follows that $\a^{-1}$ is continuous as well and thus it is a homeomorphism from $[0,1]$ to $Y$.

Now we have to show that $Y$ is closed in $X$.  Let $\alg B$ be the $0$-generated subalgebra of $\prod_{r\in[0,1]}\alg R_r$.  Then the  elements of $\alg B$  are closed terms and, for $f \in B$, we write $f(r)=(a_f(r),b_f(r))$, where $a_f,b_f$ are continuous. Let $M$ be a maximal proper filter of $\alg B$. The closed sets
	\[
	Z_f=\{r:a_f(r)=1\},\qquad f\in M,
	\]
	have the finite intersection property. Otherwise, for $f=f_1\meet\cdots\meet f_m\in M$, compactness of $[0,1]$ would give $a_f(r)\le\rho<1$ for all $r$, and therefore $f^N=\bot$ for some $N$, contradicting the properness of $M$. Since $[0,1]$ is compact, $\bigcap_{f\in M}Z_f \ne \emptyset$. Choose then $r \in \bigcap_{f\in M}Z_f$ and let
	\[
	E_r=\{f\in B:f(r)\ge1_r\}.
	\]
Then $E_r$ is a filter of $\alg B$ and $M \sse E_r$. Consider now the projection homomorphism
$$
\pi_r: \alg B\longrightarrow \alg R_r;
$$
since  $\alg R_r$ is $0$-generated, $\pi_r$ is an epimorphism whose associated congruence filter is $E_r$. Since $\alg R_r$ is simple, $E_r$ must be maximal
and thus $E_r=M$. So the maximal filters of $\alg B$ are exactly the congruence filters associated with the projections.

Finally, let $D$ be the filter of $\alg F$ corresponding to the canonical epimorphism $\alg F\to\alg B$. The maximal filters of $\alg B$ correspond exactly to the maximal filters of $\alg F$ containing $D$, that is, to the closed subspace $T(D)$ of $X$. Hence
	\[
	T(D)=\{P_r:r\in[0,1]\},
	\]
	so the image of $[0,1]$ is closed in $X$. Now the second part of the theorem follows directly from Theorem~\ref{thm:atomsupport}.

\end{proof}

\begin{corollary}\label{cor1:realinterval}  For every $r \in [0,1]$, $\alg R_r$ is finite if and only if $r\in\mathbb Q$. Moreover $\alg R_0\cong\alg{BA}_2$, $\alg R_1\cong\alg K_3$, and $\alg R_{1/(n+1)}\cong\alg C_{n+4}$ for $n\ge1$.
\end{corollary}
\begin{proof} If $r=m/n$ is rational, every coordinate generated by the constants belongs to $\{0,1/n,\ldots,1\}$, so $\alg R_r$ is finite. If $r$ is irrational, taking $j_k=\lfloor2^kr\rfloor$ in $(*)$ gives the pairwise distinct elements $\delta(2^kr-\lfloor2^kr\rfloor)$, so $\alg R_r$ is infinite. The cases $r=0,1$ give respectively $\alg{BA}_2$ and $\alg K_3$. For $r=1/(n+1)$ the $0$-generated universe is
	\[
	\{\bot,0_r,\delta(r),\ldots,\delta(nr),1_r,\top\};
	\]
	under $\delta(ir)\mapsto a_i$ its order and multiplication are exactly those defining $\alg C_{n+4}$ above.
\end{proof}	

\begin{corollary}\label{cor2:realinterval}  In the lattice $\Lambda(\mathsf{GA})$ there are exactly $2^{\aleph_0}$ atoms. In fact the countable family  $\{\vv U_q:q\in\mathbb Q\cap[0,1]\}$ consists of atoms generated by finite distributive algebras and its join has all the  $\vv U_r$, $r\in[0,1]$, as atoms. If $r$ is irrational,  $\vv U_r$ contains no nontrivial finite algebra.
\end{corollary}
\begin{proof}   There are at most $2^{\aleph_0}$ atoms in $\Lambda(\mathsf{GA})$, since the $0$-generated free algebra is countable. The second statement  follows from Theorem \ref{thm:realinterval}, since the rationals are dense in $[0,1]$.

Finally, if $r$ is irrational and $\alg A\in\vv U_r$ is nontrivial, Corollary~\ref{cor:freezero} gives a homomorphism $\alg R_r\to\alg A$ with nontrivial image; simplicity makes it injective, so $\alg A$ cannot be finite.
\end{proof}

\begin{corollary}\label{cor3:realinterval} For all $r \in [0,1]$,  $\vv U_r$ is neither compact nor completely join-prime in $\Lambda(\mathsf{GA})$.
\end{corollary}

\begin{proof}
	For fixed $r \in [0,1]$ the set $(\mathbb Q\cap[0,1])\setminus\{r\}$ is dense. By Theorem \ref{thm:realinterval}, $\vv U_r$ lies below the join of the corresponding atoms but not below any finite subjoin, and below none of them individually. Thus $\vv U_r$ is noncompact and not completely join-prime.
\end{proof}

\section{Final remarks}

\begin{remark} A logic has the {\em finite model property} if for every non provable formula $\f$ there is a finite model in which $\f$ does not hold.
Through the Blok-Pigozzi connection it is clear that if a logic $\cc L$ is algebraizable with equivalent variety semantics $\vv V$, then $\cc L$ has the finite model property if and only if $\vv V$ is generated by its finite algebras.

Now it follows from Harrop's result in \cite{Harrop1956} that if a logic $\cc L$ is finitely axiomatizable and has the finite model property then it is decidable
(i.e. $\cc L$, viewed as a set of sentences, is recursive). This because finite axiomatizability implies that $\cc L$ is recursively enumerable, while the finite model property guarantees that the complement of $\cc L$ is recursively enumerable. Hence, by Post's classical result, $\cc L$ is recursive.  If $\cc L$ is algebraizable with equivalent variety semantics $\vv V$, then $\cc L$ is finitely axiomatizable and has the finite model property if and only if $\vv V$ is finitely axiomatizable and generated by its finite algebras. Now $\cc{MALL}$ has the finite model property, as shown by Lafont in \cite{Lafont1997} using relational models; since it is finitely axiomatizable it is decidable and moreover the variety of Girard algebras is generated by its finite algebras.

In contrast, $\cc{LL}$ is finitely axiomatizable but undecidable \cite{Lincolnetal1992}. It follows that it cannot have the finite model property and thus the variety of girales is not generated by its finite algebras. Hence there is an equation in the language of girales which holds in every finite girale but fails to hold in general. This equation can be found by translating the expression considered in Theorem 3 of \cite{Lafont1997}.
\end{remark}

\begin{remark} A stronger property is the following: a logic $\cc L$ has the {\em strong finite model property} if every derivation that fails in $\cc L$ fails in some finite models. It can be shown that if $\cc L$ is algebraizable and its equivalent variety semantics $\vv V$ has finite type (i.e. finitely many fundamental operations) then the strong finite model property is equivalent to the following property: for any algebra $\alg B \in \vv V$ and any finite partial subalgebra $\alg A$ of $\alg B$, $\alg A$ is embeddable in a finite $\alg C \in \vv V$.  This property has sense regardless of the fact that $\vv V$ is the equivalent algebraic semantics and was originally formulated in a pure algebraic setting as the {\em finite embeddability property}.
It turns out that a variety of finite type has the finite embeddability property  if and only if it is generated as a quasivariety by its finite algebras.  An argument very similar to the one above shows that a finitely axiomatized quasivariety of finite type with the finite embeddability property has a decidable universal theory.

 Girard algebras do not have the finite embeddability property. Let $\alg Z$
be the infinite cyclic group generated by an element $0$; thus
$Z=\{0^n:n\in\mathbb Z\}$. Set  $0^0 :=1$; then define $0^n0^m =0^{n+m}$ and $0^n \imp 0^m = 0^{m-n}$. Hence $\cdot$ and $\imp$ form a residuated pair w.r.t. the total ordering $0^n \le 0^m$ if $n\le m$.

Moreover $(0^n\imp0)\imp0=0^{1-(1-n)}=0^n$, so $\alg Z$ is an involutive $\mathsf{FL}_e$-algebra with dualizing element $0$. Let $Z_\bot=Z\cup\{\bot,\top\}$, ordered by $\bot<0^n<\top\qquad(n\in\mathbb Z)$. Extend multiplication by $\bot x=\bot\quad(x\in Z_\bot),\qquad\top x=\top\quad(x\ne\bot)$, and extend $\imp$ to $Z_\bot$ in the following way:
\begin{align*}
& \top\imp x = \left\{
             \begin{array}{ll}
               \bot, & \hbox{if $x \in Z \cup \{\bot\}$;} \\
               \top, & \hbox{if $x=\top$}
             \end{array}
           \right.\\
&x\imp \top = \top\quad\text{for all $x \in Z_\bot$}\\
&\bot\imp x = \top\quad\text{for all $x \in Z_\bot$}\\
&x\imp \bot= \left\{
             \begin{array}{ll}\
               \bot, & \hbox{if $x \in Z \cup \{\top\}$;} \\
               \top, & \hbox{if $x=\bot$.}
             \end{array}
           \right.\\
\end{align*}
 It is an easy exercise to check that $\alg Z_\bot$ is a bounded residuated lattice. Moreover $(\top\imp0)\imp0=\top,\qquad(\bot\imp0)\imp0=\bot$, so $\alg Z_\bot$ is a Girard algebra. Since $0^{-1}=0\imp1$, it is also $0$-generated.

Now $\alg Z_\bot$ fails to satisfy the quasiequation
$$
x \ge 1\ \text{and}\ xy\app 1  \Longrightarrow  x=1
$$
which says that the only invertible element greater than $1$ is $1$. Every finite Girard algebra satisfies the quasiequation though; if $\alg A$ is a finite Girard algebra $a \in A$, $a > 1$ and $ab=1$, then $a^nb^n=1$ so $a^n$ is invertible for all $n$. Moreover $a^n < a^{n+1}$ for all $n$; if not, since $a > 1$,
$a^n = a^{n+1}$ and thus
$$
1=a^{n+1}b^n= aa^nb^n =a,
 $$
 a clear contradiction. Thus $1<a<a^2<a^3<\cdots$, which is impossible since $\alg A$ is finite.  It follows that $\mathsf{GA}$ does not have the finite embeddability property, since  the quasivariety generated by the finite Girard algebras is strictly smaller than $\mathsf{GA}$.
 Since there are not many varieties coming from logic with this property, this quasivariety might be worth exploring.

\end{remark}

\begin{remark}  A {\bf splitting pair} of a lattice $\alg L$ is a pair of elements $(a,b)$ such that $b\nleq a$ and, for every $c\in L$, either $c\le a$ or $b\le c$. Splitting pairs were introduced in \cite{Whitman1943} and used by McKenzie in his investigations on varieties of lattices \cite{McKenzie1972}. For a modern account the reader may consult \cite{AglianoCitkin2026}. In a complete lattice, $(a,b)$ is a splitting pair if and only if $a$ is completely meet-prime and $b=\bigwedge\{c\in L:c\nleq a\}$; equivalently, if and only if $b$ is completely join-prime and $a=\bigvee\{c\in L:b\nleq c\}$. If the complete lattice happens to be the lattice of subvarieties of a congruence distributive variety $\vv V$, then any splitting pair in $\Lambda(\vv V)$ is of the form $(\vv V/\alg A,\VV(\alg A))$, where
 \begin{enumerate}
 \ib $\alg A$ is  finitely generated and  subdirectly irreducible;
 \ib $\vv V/\alg A= \bigvee\{\VV(\alg B): \alg A \notin \VV(\alg B)\}$;
  \ib if $\vv V$ is generated by its finite members, then $\alg A$ is finite; moreover it is uniquely determined up to isomorphism by the splitting pair.
 \end{enumerate}.
 In the latter case $\alg A$ is called a {\bf splitting algebra} for $\vv V$. A modern account of these facts can be found in \cite{Agliano2018a}, Section 2.

 Splittings have also been applied to the lattice of axiomatic extensions of a logic (where of course the definition is dual); through the Blok-Pigozzi connection if $\cc L$ is algebraizable with equivalent variety semantics $\vv V$ then  a splitting pair in $\Lambda(\vv V)$ corresponds to a splitting pair in the lattice of axiomatic extensions of $\cc L$. This fact has been exploited at length for varieties of normal modal algebras and the corresponding lattice of logics $\cc {NExtK}$ (see \cite{Kracht1999}, Chapter 7). Also a lot is known about varieties of $\mathsf{FL}$-algebras (or residuated lattices) and the corresponding logics (\cite{Agliano2017c}, \cite{Agliano2018a}, \cite{Agliano2018b}, \cite{AglianoUgolini2019a}).

 Girard algebras, though adjacent to involutive $\mathsf{FL}_e$-algebras are very much unexplored.   Corollary \ref{cor3:realinterval} here shows that all the atoms of the form $\vv U_r$, $r\in [0,1]$ are not splitting  since they are not completely join-prime. The complexity of the lattice and this negative result may lead one to suspect that there are not  any splitting at all. Though disappointing, this would be interesting to know so we pose it as an open problem.
 \end{remark}

\section{Conclusions and future work}
 We have developed several aspects of the algebraic theory of Girard algebras, beginning with a correction of the completion argument in \cite{Agliano2025} and a direct proof of the conservativity of $\cc{LL}$ over $\cc{MALL}$. The main structural results concern the bottom of the subvariety lattice $\Lambda(\mathsf{GA})$. For every Girard algebra $\alg A$, its $0$-generated subalgebra is the free algebra on the empty set of $\VV(\alg A)$; consequently, simple $0$-generated Girard algebras correspond exactly to the atoms of $\Lambda(\mathsf{GA})$, and every nontrivial subvariety contains an atom. These atoms can be represented uniformly by the maximal filters of the free $0$-generated algebra. Moreover every congruence of this free algebra is realized by a subvariety, and the resulting atom-support topology records the behaviour of arbitrary joins: the atoms below a join of atoms are precisely the points in the closure of the corresponding subset of the spectrum.

The concrete constructions show that the bottom of $\Lambda(\mathsf{GA})$ is already quite rich. Besides determining the integral and degenerate boundary cases and exhibiting finite modular and nonmodular atoms and chains of covers, we constructed infinite families of atoms whose topology reflects genuinely infinitary behavior. A cyclic family yields a subspace homeomorphic to $\omega+1$ and, in particular, a noncompact atom. More substantially, the atom spectrum contains a closed copy of $[0,1]$ on which topological closure agrees exactly with the formation of atoms below joins. It follows that $\Lambda(\mathsf{GA})$ has exactly $2^{\aleph_0}$ atoms, and that a countable family of atoms generated by finite distributive algebras can have a join containing continuum many atoms. Thus finite and infinite joins behave essentially differently at the bottom of the subvariety lattice. On the logical side, these results yield $2^{\aleph_0}$ Post-complete extensions of $\cc{MALL}$.

The most immediate continuation of this work concerns girales. In subsequent work we will study the effect of the exponential operation on the subvariety structure, with particular emphasis on simple and $0$-generated girales, the Heyting algebra of open elements, and the relation between a girale, its Girard reduct, and its associated Heyting algebra. These considerations lead naturally to the study of linear companions, the fibers induced on the corresponding subvariety lattices by the reduct and Heyting constructions, and the interaction between the involutive and Heyting negations.

A further problem is to determine whether the topological description developed here for the atom spectrum of $\Lambda(\mathsf{GA})$ has an analogue for girales. In particular, it would be interesting to understand how such a spectrum interacts with the Girard-reduct and Heyting constructions, and how its topology reflects algebraic properties such as finite generation, compactness and complete join-primality. These questions will be pursued in a sequel devoted to girales.

\bibliographystyle{amsplain}
\bibliography{further}

\end{document}